\documentclass[11pt]{article}
\usepackage{latexsym}

\usepackage[top=1in, bottom=1in, left=1in, right=1in]{geometry}

\usepackage{adjustbox}

\usepackage{amssymb}

\usepackage{graphicx} 
\usepackage{color} 
\usepackage{amsmath, amsthm, amssymb}
\usepackage{enumerate}
\usepackage{float}
\usepackage{url}
\usepackage{stackrel}
\usepackage{mathrsfs,dsfont}
\usepackage{caption}
\usepackage{subcaption}
\usepackage{wrapfig}
\usepackage{bbm}
\usepackage{bm}
\usepackage{epigraph}
\usepackage{physics}
\usepackage[colorinlistoftodos, shadow]{todonotes} 

\newtheorem{theorem}{Theorem}[section]
\newtheorem{corollary}[theorem]{Corollary}
\newtheorem{proposition}[theorem]{Proposition}
\newtheorem{lemma}[theorem]{Lemma}
\newtheorem{definition}[theorem]{Definition}
\newtheorem{example}[theorem]{Example}

\newtheorem{remark}[theorem]{Remark}

\newcommand{\been}{\begin{enumerate}}
\newcommand{\enen}{\end{enumerate}}
\newcommand{\beit}{\begin{itemize}}
\newcommand{\enit}{\end{itemize}}

\def\CC{\mathcal C}

\def\II{\mathcal I}
\def\PP{\mathcal P}

\def\HH{\mathcal{H}}

\def\DD{\mathcal{D}}

\def\BB{\mathcal B}

\def\one{{\mathds 1}}
\def\diag{{\rm diag}}
\def\spn{{\rm span}}

\def\vect{\mathbf}

\def\ep{\epsilon}

\newcommand{\R}{\mathbb{R}}

\newcommand{\Z}{\mathbb{Z}}
\newcommand{\C}{\mathbb{C}}

\usepackage{tikz}    
\usetikzlibrary{shapes,automata,positioning,arrows,fit}
\usepackage{mathtools}  
\usetikzlibrary{snakes}

\usepackage{caption}
\usepackage[font=scriptsize]{caption}
\usepackage{mwe}

\newcommand{\specialcell}[2][c]{\begin{tabular}[#1]{@{}c@{}}#2\end{tabular}}

\numberwithin{equation}{section}

\usepackage{tikz-3dplot}

\usepackage{relsize}

\newcommand\cca[1]{{%
  \ooalign{\raisebox{-.6ex}{\larger[12]$\circlearrowright$}\cr
    \hidewidth$\,#1$\hidewidth}}}

 \tikzset{every node/.style={auto}}
 \tikzset{every state/.style={rectangle, minimum size=0pt, draw=none, font=\normalsize}}
  \tikzset{bend angle=7}
  
  \usepackage{array}
  
  \def\wt{\widetilde}
    
  \def\ol{\overline}
  
  \def\bd{\mathrm{bd}~}
  
\title{Autocatalytic Networks:  \\ An Intimate Relation between Network Topology and Dynamics}
\author{Badal Joshi\footnotemark[1] \and Gheorghe Craciun\footnotemark[2]}

\begin{document}

 \footnotetext[1]{Department of Mathematics, California State University San Marcos.}
 \footnotetext[2]{Departments of Mathematics and Biomolecular Chemistry, University of Wisconsin-Madison.}

\date{}



\maketitle

\begin{abstract}

We study a family of networks of autocatalytic reactions, which we call hyperchains, that are a  generalization of hypercycles. Hyperchains, and the associated dynamical system called replicator equations, are a possible mechanism for macromolecular evolution and proposed to play a role in abiogenesis, the origin of life from prebiotic chemistry. The same dynamical system also occurs in evolutionary game dynamics, genetic selection, and as Lotka-Volterra equations of ecology. 
An arrow in a hyperchain encapsulates the enzymatic influence of one species on the autocatalytic replication of another. 
We show that the network topology of a hyperchain, which captures all such enzymatic influences, is intimately related to the dynamical properties of the mass action system it generates. 
Dynamical properties such as existence, uniqueness and stability of a positive equilibrium as well as permanence, are determined by graph-theoretic properties such as existence of a spanning linear subgraph, being unrooted, being cyclic, and Hamiltonicity. 
\end{abstract}

\section{Introduction}

We study the deterministic dynamics of {\em autocatalytic reaction networks}, where every reaction has the form $X_i + X_j \longrightarrow 2X_i +X_j$, where $i$ may or may not be different from $j$. 
The molecule $X_j$ is a catalyst for this reaction since its concentration is unaffected by the reaction. $X_i$ is catalyzed, in the sense that $X_i$ acts as a template for manufacture of a copy of itself.  A molecule can be a catalyst in one reaction and a template in another, or play both roles in the same reaction.
Such a network of {\em template-induced enzyme-catalyzed replication} was introduced by the Nobel laureate Manfred Eigen \cite{eigen1971selforganization} and further developed in a three part series of papers by Eigen and Schuster \cite{eigen1977emergence, eigen1978abstract, eigen1978realistic}; see also the collection in book form \cite{eigen1979hypercycle}.  Eigen and Schuster depicted the autocatalytic replication of $X_i$ via $X_j$, i.e. the reaction $X_i + X_j \longrightarrow 2X_i +X_j$, through the diagram $X_j \longrightarrow \cca{X_i}$. We preserve the principle of the notation, but for convenience replace it with $X_j \dashrightarrow X_i$, avoiding the solid arrow because it clashes with the arrow used to depict a reaction. The set of autocatalytic reactions proposed by Eigen had a cycle structure in the graph of catalytic influences, and was therefore dubbed a {\em hypercycle}. 

The motivation for introducing a hypercycle was to provide a mechanism to thwart the {\em error threshold} problem. This problem arises in {\em abiogenesis}, the origin of biotic life from prebiotic chemistry. For Darwinian natural selection to find purchase, there must be an information-carrying molecule that metabolizes and replicates faithfully over time. If there are frequent copying errors, the information will dissipate over time. In absence of error-correcting mechanisms, this sets a strict upper bound on the size of the molecule. Modern cells get around this size restriction by deploying enzymes that perform error correction. However, to produce these large error-correcting proteins, the coding molecule must have large size, which creates a classic chicken-and-egg dilemma. It appears that high fidelity replication of large molecules can only occur if large molecules that code for error-correction already exist. 
The path to producing large genomic molecules from small ones has not been elucidated. Hypercycles are a potential solution to this seeming paradox, since they are self-maintaining ecosystems that carry information in a robust manner. When there are competing hypercycles, even the slightest initial advantage for one of them becomes a winner-take-all situation, where only one connected motif survives in the long run. 
A hypercycle allows a diversity of molecule types to coexist and collaborate, effectively forming a molecular ecosystem, which is far more robust than a single molecule type, or a motley collection
of competing molecular species. It was speculated by Eigen that this system may have the ability to evolve and become more complex by allowing addition of new molecule types to the ecosystem. Since a community of molecules in a hypercycle can maintain and increase their information content by evolving, hypercycle may be a candidate for a solution to the error threshold problem. 

\begin{figure}[h!] 
 \begin{center}
   \begin{tikzpicture}[auto, every node/.style={scale=1}]
        \tikzstyle{block} = [draw, rectangle];
        \node [block] (cycle) {$\HH$ cyclic};
        	\node [block, right=11cm of cycle] (permanent) {$\HH$ permanent};
        \node [block, below=1.5cm of cycle] (hamiltonian) {$\HH$ Hamiltonian};
	\node [block, below=1.5cm of permanent] (perm sys) {$(\HH,K)$ permanent};
	\node [block, below left=0.6cm of perm sys] (persistent) {\specialcell{$(\HH,K)$ persistent and has \\ unique positive equilibrium}};
	\node [block, left=1.8cm of persistent] (irreducible) {\specialcell{$\HH$  strongly \\ connected}};
		\node [block, below=0.5cm of persistent] (stable) {\specialcell{$(\HH,K)$ has a linearly stable \\ positive equilibrium}};
	\node [block, below=0.8cm of irreducible] (property) {$(\HH,K)$ has property P};
	\node [block, below=3.7cm of hamiltonian] (lin sub) {\specialcell{$\HH$ has a  spanning \\ linear subgraph}};
	\node [block, below=1.8cm of lin sub] (unroot) {$\HH$ is unrooted};
	\node [block, below=3.6cm of perm sys] (uniq pos) {\specialcell{$(\HH,K)$ has unique \\ positive equilibrium}};
	\node [block, below=1.4cm of uniq pos] (pos) {\specialcell{$(\HH,K)$ has a \\ positive equilibrium}};
	
        
          \draw[<->,double] (cycle)--(permanent);
          \draw[->,double] (cycle)--(hamiltonian);
          \draw[->,double] (permanent)--(perm sys);
          \draw[->,double] (hamiltonian)--(perm sys);
          \draw[->,double] (hamiltonian)--(irreducible);
          \draw[->,double] (perm sys)--(persistent);
          \draw[->,double] (persistent)--(irreducible);
          \draw[->,double] (hamiltonian)--(lin sub);
          \draw[->,double] (lin sub)--(unroot);
          \draw[->,double] (perm sys)--(uniq pos);
          \draw[->,double] (uniq pos)--(pos);
          \draw[<->,double] (lin sub)--(uniq pos);
          \draw[<->,double] (unroot)--(pos);
          
           \draw[<->,double] (stable)--(property);
           \draw[->,double] (stable)--(uniq pos);
           \draw[->,double] (property)--(lin sub);
          
         
   \end{tikzpicture} 
 \end{center}
 \caption{
This table gives a summary of the most important results in this paper. A central goal of this paper is to connect on the one hand the graph-theoretic properties of a certain network, called the hyperchain and denoted by $\HH$, associated with a set of autocatalytic reactions and on the other hand the dynamical system $(\HH,K)$ arising from pairing with a choice of mass action kinetics $K$. The six boxes on the left are graph-theoretic properties of $\HH$, and the six boxes on the right are dynamical properties of $(\HH, K)$.
The symbol ``$\implies$'' indicates an implication.  
``$\HH$ is permanent'' is interpreted as ``$(\HH,K)$ is permanent for any $K > 0$'', while ``$(\HH,K)$ is permanent'' is interpreted as ``$(\HH,K)$ is permanent for some $K > 0$''. Property P is ``(i) $K$ is invertible, (ii) $(K^T)^{-1} \one > 0$, and (iii) $\diag((K^T)^{-1} \one) K^T$ has $n-1$ eigenvalues with negative real part.''}
 \label{fig:sum_implications}
 \end{figure}
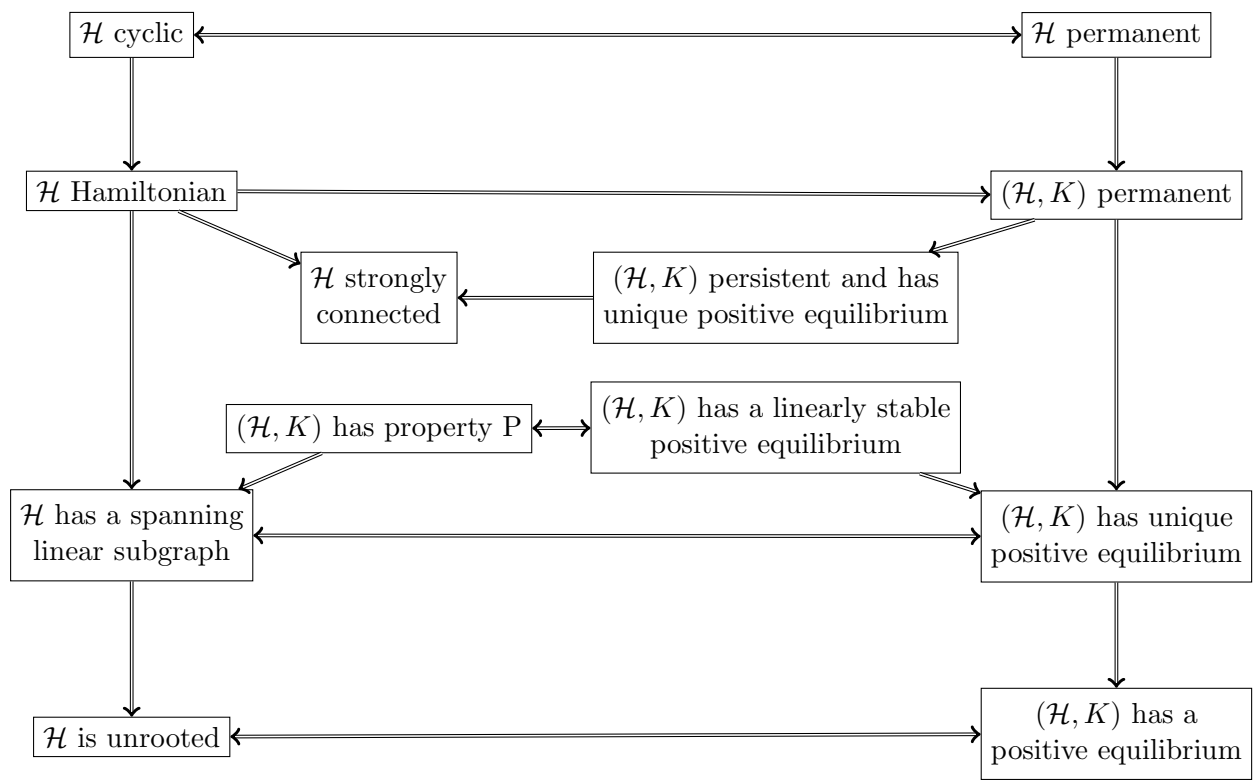


Hyperchains are a clearly defined set of networks, and include hypercycles as a subset. 
The original formulation of hypercycles involved ODEs and was studied extensively in \cite{hofbauer1988theory,schuster1978dynamical,hofbauer1980dynamical,schuster1979dynamical,hofbauer1981general,hofbauer1981competition,hofbauer1991stable,hofbauer2002competitive,hofbauer1998evolutionary}. 
The system of ODEs associated to hyperchains, referred to as replicator equations in \cite{hofbauer1988theory}, occurs not only in macromolecular evolution, but also in evolutionary game dynamics, and in Fisher-type genetic selection \cite{hofbauer1988theory}. Furthermore, Lotka-Volterra systems of mathematical ecology \cite{takeuchi1996global} are equivalent to hyperchain systems. It is shown in Theorem 7.5.1 of \cite{hofbauer1998evolutionary} that orbits of any Lotka-Volterra system are in a one-to-one correspondence with orbits of a hyperchain system after a coordinate transformation. 
Even though ``replicator equations'' is established terminology for the dynamical system associated to a hyperchain, we avoid its use in this article because we wish to emphasize the graph-theoretic properties of catalytic influences and how these affect the dynamical properties. In that sense, the nomenclature ``hyperchain'' is a natural successor to ``hypercycle''. 

Long-term survival of a community of molecules is aptly captured in the mathematical property of permanence. The classic hypercycle with $n \ge 2$ species is known to be permanent for {\em all} rate constants \cite{hofbauer1988theory}. When one considers other topologies of catalytic influences besides a cycle, some are found to be permanent while others are not, and the permanence property further depends on the choice of rate constants. 
We refer to the  generalization of hypercycle to arbitrary topologies of autocatalytic reactions of the form $X_i + X_j \to 2X_i + X_j$  by {\em hyperchain}. 
The fundamental mathematical questions we address here are: which hyperchains are permanent, and what are their dynamical properties. Besides permanence, we focus on existence, uniqueness, and stability of positive equilibria and of boundary equilibria. We give fairly general results that connect the network properties of the hyperchain with the dynamical properties of the mass action dynamical system generated by the hyperchain. 

We now give a highlights tour of the main mathematical results in this article. 
Several of these results, among others, are summarized in the network of implications in Figure \ref{fig:sum_implications}.
We establish that all hyperchain systems have unbounded growth, and a hyperchain system does not have a finite-time blow-up if and only if it is acyclic (Theorem \ref{thm:fintimeblowup}). Moving on to the dynamics of relative concentrations, we prove the converse of the well-known result related to permanence of hypercycles. We show that if a hyperchain is permanent for all rate constants, then it must be a hypercycle (Theorem \ref{thm:perm_then_cycle}). While the Hamiltonian property is not necessary for permanence of a hyperchain system, we show that it is sufficient (Theorem \ref{thm:hamilton_perm}). A hyperchain has a  spanning linear subgraph if and only if there exists a dynamical system generated by the hyperchain which has a unique positive equilibrium (Theorem \ref{thm:uniqueposeq}). 

This article is organized as follows.  
In Section \ref{sec:networktopology}, we discuss the network properties of a hyperchain and establish the basic notation. In Section \ref{sec:dynamics}, we discuss the dynamics of the system of ODEs generated by applying mass action kinetics to a hyperchain. In Section \ref{sec:relativeconcentration}, we discuss the dynamical system of relative concentrations (relative to the total concentration) of species, and from hereon discuss only the relative concentration system. In Section \ref{sec:dynamics_relconc}, we establish the network conditions for existence, uniqueness, and stability of equilibria of a hyperchain system. Finally, in Section \ref{sec:permanence}, we discuss the network conditions for permanence of a hyperchain system.


\section{Network Topology of a Hyperchain} \label{sec:networktopology}

We consider a reaction network with species $\{X_1, \ldots, X_n\}$
 where every reaction is of the form $X_i + X_j \to 2X_i +X_j$.
Each reaction models  {\em template-induced enzyme-catalyzed replication}, where $X_i$ is the template or the replicating species, and $X_j$ is the enzyme. A species can appear as a replicator in one reaction and as an enzyme in another. The {\em catalytic influence diagram} $X_j \dashrightarrow {X_i}$, which represents the reaction $X_i + X_j \to 2X_i +X_j$, depicts the catalytic species at the tail of a dashed arrow and the replicating species at the head of the same arrow. 
Any network containing only reactions of type $X_i + X_j \to 2X_i +X_j$ 
can be mapped in a one-to-one manner to a {\em catalytic influence network (CIN)} by mapping each reaction in the network to its catalytic influence diagram and then taking the union of such diagrams. We emphasize that this is a one-to-one mapping and the mapping can be reversed for any given CIN to obtain a unique reaction network. We refer to the CIN thus obtained as a {\em hyperchain}. 

%
\begin{definition}
Let $V = \{X_1, \ldots, X_n\}$ be a finite, nonempty set and let $D$ be a nonempty subset of $V \times V$
with the property that for every $X_i \in V$ there is an $X_j \in V$ such that either $X_i \dasharrow X_j \in D$ or $X_j \dasharrow X_i \in D$. We will refer to the directed graph $\HH=(V,D)$ as a {\em hyperchain on $\{X_1, \ldots, X_n\}$} or as a {\em hyperchain on $n$ species}. 
\end{definition}
Unless mentioned otherwise, we will assume that the set of species (or set of vertices) of the hyperchain is $\{X_1,\ldots, X_n\}$. From hereon, we suppress explicit mention of the sets $V$ and $D$ in the description of a hyperchain $\HH = (V,D)$, and using a slight abuse of notation we say $X_i \dasharrow X_j \in \HH$ to mean $X_i \dasharrow X_j \in D$.


Consider a hyperchain on 6 species with cyclic catalytic influence, as depicted in \eqref{fig:hyp6}. We will refer to a hyperchain on $n$ species with a single cycle as a {\em $n$-hypercycle}, or simply as {\em hypercycle}, when the statement is about arbitrary number of species. The name hypercycle was coined by Eigen in seminal work on the origin of biological macromolecules \cite{eigen1971selforganization}. In a hypercycle, every species is both a replicator and an enzyme for exactly one other species and furthermore, every species directly or indirectly (i.e. possibly through a sequence of arrows) aids the replication of every other species. This {\em pay-it-forward} form of cooperation within a single community of replicators is responsible for some very robust dynamical properties. In fact, the goal of this paper is to establish that there is an intimate connection between the graph-theoretic properties of a hyperchain on the one hand, and the dynamical properties of the mass-action system for the same network on the other hand. 
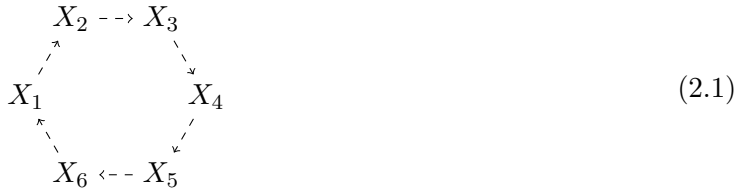
\begin{equation} \label{fig:hyp6}
  \begin{tikzpicture}[baseline={(current bounding box.center)}, scale=1.2]
   \node[state] (1)  at (cos 180,sin 180)  {${X_1}$};
   \node[state] (2)  at (cos 120,sin 120)  {${X_2}$};
   \node[state] (3)  at (cos 60,sin 60)  {${X_3}$};
   \node[state] (4)  at (cos 0,sin 0)  {${X_4}$};
   \node[state] (5)  at (cos -60,sin -60)  {${X_5}$};
   \node[state] (6)  at (cos -120,sin -120)  {${X_6}$};

   \path[->,dashed]
    (1) edge[] node {} (2)
   
     (2) edge[] node {} (3)
     
    (3) edge[] node {} (4)
   
     (4) edge[] node {} (5)
     
     (5) edge[] node {} (6)
     
         (6) edge[] node {} (1)
     
;
  \end{tikzpicture}
 \end{equation}
 
Consider a hyperchain $\HH$ on species $\{X_1, \ldots, X_n\}$. Let $x_i$ represent the time-dependent concentration of the species $X_i$. We will assume mass-action kinetics throughout, so that the rate of the reaction $X_i + X_j \to 2X_i +X_j$ is proportional to the product of the concentrations of the reactants, i.e. to $x_i x_j$. The rate of proportionality is the mass-action reaction rate constant, denoted by $k_{ji}$. When referring to the mass action system, we will use the reaction rate constant as a label on the corresponding edge in the hyperchain, as depicted below.
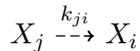
\begin{equation*}
X_j \stackrel{k_{ji}}{\dashrightarrow} {X_i} 
\end{equation*}
Thus the mass-action system of $\HH$ is specified via a labeled directed graph, where the labels (or weights) are positive constants. For convenience in writing some mathematical expressions, we allow the weights to be zero, i.e. $k_{ji}$ is equal to zero if and only if $X_j \dashrightarrow X_i$ is not an edge in the hyperchain. With this convention, the mass-action ODE system governing the dynamics of a hyperchain is
\begin{equation} \label{eq:absconcode}
\dot x_j = x_j \sum_{i=1}^n k_{ij} x_i = x_j f_j(x) \quad \quad (1\le j \le n),
\end{equation}
where $f_j(x) := \sum_{i=1}^n k_{ij} x_i$. In vector notation, $\dot x = x \ast f(x)$ where $\ast$ denotes the termwise product of vectors. 


Consider a hyperchain $\HH$ on $\{X_1,\ldots,X_n\}$ taken with mass action kinetics. Let $A(\HH)$ denote the adjacency matrix of $\HH$, i.e. 
\begin{align*}
(A(\HH))_{i,j} = \begin{cases}
1 & \mbox{ if } X_i \dasharrow X_j \in \HH \\
0 & \mbox{ if } X_i \dasharrow X_j \notin \HH. 
\end{cases}
\end{align*}
We define a matrix $K \in \R^{n \times n}_{\ge 0}$ whose entries are the mass action rate constants,  via
\begin{align*}
(K)_{i,j} = \begin{cases}
k_{ij} > 0 & \mbox{ if } X_i \dasharrow X_j \in \HH \\
0 & \mbox{ if } X_i \dasharrow X_j \notin \HH. 
\end{cases}
\end{align*}
Let $(\HH,K)$ be a labeled, directed graph, where the labels on $\HH$ are provided by $K$. Often, we need to consider a {\em sub-hyperchain} $\HH'$, i.e. a hyperchain $\HH'$ that is a subgraph of a hyperchain $\HH$. Define the restriction $K |_{\HH'} =  K \ast A(\HH')$, where $\ast$ denotes the entrywise product of matrices. When $\HH'$ is a sub-hyperchain of a hyperchain $\HH$, we write $(\HH',K)$ as an abbreviation of $(\HH',K |_{\HH'})$.

For $x = (x_1,\ldots,x_n)^T$, the mass action dynamical system \eqref{eq:absconcode} in matrix notation is 
\begin{equation}
\dot x = x \ast (K^T x).
\end{equation}
By $x$ is positive (nonnegative), we mean that $x \in \R^n_{> 0}$ ($x \in \R^n_{\ge 0}$). 
\section{Dynamics of a Hyperchain} \label{sec:dynamics}
 We refer to a node in the hyperchain as a {\em terminal node} if it has no outgoing edges, and as an {\em initial node} if it has no incoming edges. A self-edge of the form $X_i \dashrightarrow X_i$ is considered both incoming and outgoing for $X_i$. Two dynamical systems $\dot x = f(x)$ and $\dot x = g(x)$ for $x \in \R^n$ are said to be {\em diagonally conjugate} if there exists an $s \in \R^n_{>0}$ such that $g(s \ast x) = s \ast f(x)$. Diagonal conjugacy, as defined here, is a special form of topological conjugacy. The trajectories of diagonally conjugate systems are in one-to-one correspondence, as shown by the following Lemma.
 \begin{lemma}
 Consider the dynamical systems $\dot x = f(x)$ and $\dot x = g(x)$ for $x \in \R_{\ge 0}^n$ with the property that there is an $s \in \R^n_{>0}$ such that $g(s \ast x) = s \ast f(x)$. Then, 
 $\{x(t) : t \in \R\}$ is an orbit of  $\dot x = f(x)$ if and only if $\{s \ast x(t) : t \in \R\}$ is an orbit of  $\dot y = g(y)$. 
 \end{lemma}
\begin{proof}
Let $s = (s_1, \ldots, s_n) \in \R^n_{>0}$ such that $g(s \ast x) = s \ast f(x)$. Define $s^{\ast -1} = (1/s_1, \ldots, 1/s_n)  \in \R^n_{>0}$. If we let $y := s \ast x$, then $s^{\ast -1} \ast g(y) = f(s^{\ast -1} \ast y)$. So it suffices to prove one direction of the implication. Now, let $\{x(t) : t \in \R\}$ be an orbit of  $\dot x = f(x)$. Then $\dv{}{t} \left(s \ast x(t)\right) = s \ast \dot x = s \ast f(x) = g(s \ast x)$, so that $\{s \ast x(t) : t \in \R\}$ is an orbit of  $\dot y = g(y)$.
\end{proof}
For any given hyperchain system, we can simplify the analysis by considering a non-dimensionalized system, wherein we set several of the rate constants equal to 1, without loss in generality. 
\begin{proposition} \label{prop:nondim}
Let $\HH$ be a hyperchain on $n$ species and $K \in \R^{n \times n}_{\ge 0}$ a matrix of mass action reaction rate constants. There exists a $K_0 \in \R^{n \times n}_{\ge 0}$, with the property that every non-terminal node in the hyperchain has at least one outgoing edge with weight 1, and such that the mass action systems under $K$ and $K_0$ are diagonally conjugate. 
\end{proposition}
\begin{proof}
For $x \in \R^n_{\ge 0}$, let $\dot x = x \ast \left(K^T x\right)$ be a mass action system under $K$, where $K$ is the matrix whose $(i,j)$ entry is $k_{ij}$, so that for $1 \le i \le n$, $\dot x_i = x_i \sum_{j=1}^n k_{ji} x_j$. 
For every node $X_i$ that is not terminal, fix an outgoing edge, and denote the rate constant for this edge by $k_{ii_s}$.  For all $i, j \in \{1,\ldots, n\}$, define $\wt k_{ij} := k_{ij} k_{i i_s}^{-1}$ and let $K_0$ be the matrix whose $(i,j)$ entry is $\wt k_{ij}$.

Define $s \in \R^n_{>0}$ as follows:
\begin{equation*}
s_i = \begin{cases} k_{ii_s} & \mbox{ if $X_i$ is not terminal,} \\ 1 & \mbox{ if $X_i$ is terminal}. \end{cases}
\end{equation*}
It follows that
\begin{equation} \label{calcsss}
\dv{}{t} \left(s_i x_i(t)\right) = s_i \dot x_i = s_i x_i \sum_{j=1}^n k_{ji} x_j = s_i x_i \sum_{j=1}^n k_{ji} k_{jj_s}^{-1} k_{jj_s} x_j  = s_i x_i \sum_{j=1}^n \wt k_{ji} s_j x_j.
\end{equation}
Clearly, this defines a mass action system on the hyperchain $\HH$, where the concentration of the species $X_i$ is $s_i x_i$. Furthermore, \eqref{calcsss} shows that the mass action systems under $K$ and $K_0$ are diagonally conjugate. Finally, for every non-terminal node $X_i$, $\wt k_{ii_s} = 1$, by definition. 
\end{proof}

\begin{proposition}
Consider a hyperchain $\HH$ on $n$ species. Then, $x(t) = (x_1(t), \ldots x_n(t)) \to \infty$, in finite or infinite time, for every initial condition $x(0) \in \R^{n}_{> 0}$. 
\end{proposition}
\begin{proof}
From \eqref{eq:absconcode}, for all $j \in \{1,\ldots, n\}$, $\dot x_j = x_j \sum_{i=1}^n k_{ij} x_i \ge 0$, so that $x_j(t)$ is a nondecreasing function of $t$. In particular, $x_j(t) \ge x_j(0) > 0$ for all $t \in \R_{\ge 0}$. But then $\dot x_j \ge x_j \sum_{i=1}^n k_{ij} x_i(0)$. 
By definition, $\HH$ has at least one node that is not initial, because otherwise $\HH$ has no directed edge. Suppose that $X_j$ is non-initial in $\HH$, so there is a $k_{ij} > 0$ for some $i \in \{1,\ldots,n\}$. In particular, $\dot x_j \to \infty$ in either finite or infinite time. 
\end{proof}

Positive feedback is necessary and sufficient for finite-time blow-up in hyperchain systems, as the next results show. 
\begin{lemma} \label{lem:int-ln}
Let $\lambda: \R_{\ge 0} \to \R_{> 0}$ be a continuous, non-decreasing function such that $\lambda(t) \xrightarrow{t \to \infty} \infty$. Then the set 
\[
S := \left\{ t \ge 0 : \int_{0}^t \lambda(s) ds \le \ln \lambda(t) \right \}
\]
has finite measure.
\end{lemma}
\begin{proof}
Suppose $\mu: \R_{> 0} \to \R_{> 0}$ is a non-decreasing function such that $\mu(t) \xrightarrow{t \to \infty} \infty$ and $\int_{0}^t \mu(s) ds \le \ln \mu(t)$ for all $t > 0$. For $t_0 \in (0, \infty)$, let $\mu_0 = \mu(t_0)$. Since $\mu(t)$ is non-decreasing and $\mu(t) \xrightarrow{t \to \infty} \infty$, for $i \in \Z_{\ge 0}$, $\{ t' : \mu(t') \ge 2^i \mu_0\}$ is nonempty. Let $t_i := \inf \{ t': \mu(t') \ge 2^i \mu_0\}$  so that $0 < t_0 \le t_1 \le t_2 \le \ldots$ is a non-decreasing sequence and for all $\ep > 0$ and $i \in \Z_{>0}$, $\mu(t_i - \ep) < 2^i \mu_0 \le \mu(t_i)$. 
Let $\ep >0$. For all $i \in \Z_{\ge 0}$, 
\[
\ln(2^{i+1} \mu_0) > \ln (\mu(t_{i+1} - \ep)) \ge \int_{0}^{t_{i+1} -\ep} \mu(s) ds  \ge \int_{t_i}^{t_{i+1} -\ep} \mu(s) ds \ge \mu(t_i) \left(t_{i+1} - t_i -\ep \right), 
\]
Since $\ep$ is arbitrary, 
\[
t_{i+1} - t_i \le \frac{\ln(2^{i+1} \mu_0)}{\mu(t_i)} \le \frac{\ln(2^{i+1} \mu_0)}{2^i\mu_0}.
\] 
So for any $k \in \Z_{\ge 0}$, 
\[
t_k = t_0 + \sum_{i=0}^k \left( t_{i+1} - t_i \right) \le  t_0 + \sum_{i=0}^k \frac{\ln(2^{i+1} \mu_0)}{2^i\mu_0} <  t_0 + \sum_{i=0}^\infty \frac{\ln(2^{i+1} \mu_0)}{2^i\mu_0} = t_0 + \frac{2 \ln(4 \mu_0)}{\mu_0}
\]
This implies that $\mu(t)$ goes to infinity in finite time, which contradicts the assumption that $\mu(t)$ is defined for all positive reals. 

Now, suppose that $\lambda: \R_{\ge 0} \to \R_{> 0}$ is a continuous, non-decreasing function such that $\lambda(t) \xrightarrow{t \to \infty} \infty$ and $S := \{ t \ge 0: \int_{0}^t \lambda(s) ds \le \ln \lambda(t) \}$ has infinite measure. By continuity of $\lambda$, $S$ is a union of closed intervals, possibly including some degenerate intervals comprising of a single point. So, we may write $S = \bigcup_{\alpha \in I} [a_\alpha,b_\alpha]$ for a totally ordered index set $I$, where the total ordering on $I$ is defined as follows: for all $\alpha, \alpha' \in I$, $\alpha < \alpha'$ if and only if $a_{\alpha} < a_{\alpha'}$. 

Let $\chi_S: \R_{\ge 0} \to \{0,1\}$ be the characteristic function of the set $S$, i.e. $\chi_S(t) = 1$ if $t \in S$ and $\chi_S(t) = 0$ if $t \notin S$. 
Then, the function $t \mapsto \int_0^{t} \chi_S(s) ds$ is continuous, non-decreasing, has range $\R_{\ge 0}$ and $\int_0^{\infty} \chi_S(s) ds = + \infty$.
Now define 
\[
\varphi(t) := \inf \left \{t' \ge 0 : \int_0^{t'} \chi_S(s) ds  > t \right \}. 
\]
Then $\varphi$ is increasing, injective, has range $S' = \bigcup_{\alpha \in I} [a_\alpha,b_\alpha)$, and $\varphi(t) \xrightarrow{t \to \infty} \infty$. 
Let $\wt \lambda = \lambda \circ \varphi$.  Then $\wt \lambda: \R_{\ge 0} \to \R_{> 0}$ is a composition of non-decreasing functions that go to infinity, which implies that $\wt \lambda$ is non-decreasing and $\wt \lambda(t) \xrightarrow{t \to \infty} \infty$. 
Furthermore, for all $t > 0$, 
\begin{align*}
\ln \wt \lambda (t) &= \ln \lambda (\varphi (t)) \ge \int_{[0, \varphi(t)]} \lambda(s) ds 
 \ge \int_{[0, \varphi(t)] \cap S'} \lambda(s) ds  =  \int_{[0,t]}  \wt \lambda(s) ds, 
\end{align*}
where we used injectivity of $\varphi$ in the last equality. But, by the first part of the proof, this implies that $\wt \lambda$ is not defined for all time, which is a contradiction.
\end{proof}

\begin{corollary} \label{cor:int-ln}
Consider the ODE $\dot x = \lambda(t) x(t)$ where $t, x(t) \in \R_{\ge 0}$ and $\lambda: \R_{\ge 0} \to \R_{> 0}$ is a continuous, non-decreasing function such that $\lambda(t) \xrightarrow{t \to \infty} \infty$. Suppose that for every positive initial $x(0) \in \R_{>0}$, the ODE has a unique solution $x(t)$ for all time $t \in [0, \infty)$. Then 
$
\left\{ t \ge 0 : x(t) \le \lambda(t) \right\}
$
has finite measure.
\end{corollary}
\begin{proof}
Fix an initial condition $x(0) = x_0 > 0$. Then, by integrating the ODE, $t \in [0, \infty)$, $\ln x(t) = \int_{0}^t \lambda(s) ds$. By Lemma \ref{lem:int-ln}, the set $\{\int_{0}^t \lambda(s) ds \le \ln \lambda(t) \}$ has finite measure, and so the result follows. 
\end{proof}

\begin{theorem} \label{thm:fintimeblowup}
A hyperchain $\HH$ taken with mass action kinetics does not have a finite-time blow-up if and only if $\HH$ is acyclic.
\end{theorem}
\begin{proof}
Suppose that $\HH$ is acyclic.
We construct an increasing sequence of sub-hyperchains of $\HH$, $\HH_0 \subsetneqq \HH_1 \subsetneqq \HH_2 \subsetneqq \ldots \subsetneqq \HH_r = \HH$ as follows. Let $\HH_i = (V_i, D_i)$ where $V_i$ is the set of vertices and $D_i$ is the set of directed edges of $\HH_i$. 
Write $\HH = (V,D)$. 
Let $V_0$ be the nonempty set consisting of initial nodes of $\HH$ and let $D_0$ be the empty set. Define
\begin{align*}
V_{i+1} &:= V_i \bigcup \left \{X \in V:  X' \dashrightarrow X \mbox{ for some } X' \in V_i \right\}.
\\
D_{i+1} &:= D_i \bigcup \left \{X' \dashrightarrow X \in D:  X' \in V_i \right\}.
\end{align*}
Let $r$ be the largest index such that $\HH_{r} \setminus \HH_{r-1}$ is nonempty, i.e. $\HH_r = \HH$. For any species $X \in V_0$, $\dot x = 0$, and so $x(t) = x(0)$ for all $t \in \R_{\ge 0}$. In particular, $x(t)$ is defined for all $t \in \R_{\ge 0}$. Suppose that $y(t)$ is defined for all $t \in \R_{\ge 0}$ for every species $Y$ in $V_i$, $i \in \{0, \ldots, r-1\}$. Let $X \in V_{i+1}$. Then 

\[
\dot x(t) = x(t) \sum_{\{Y \dashrightarrow X: Y \in V_i\}} \left(k_{Y \dashrightarrow X}\right) y(t) =: x(t) \lambda(t). 
\]
Here $k_{Y \dashrightarrow X}$ is the reaction rate constant associated to $Y \dashrightarrow X$ and we defined 
\[
\lambda(t) \coloneqq \sum_{\{Y \dashrightarrow X: Y \in V_i\}} \left(k_{Y \dashrightarrow X}\right) y(t).
\] 

Clearly $\lambda(t)$ is defined for all $t \in \R_{\ge 0}$ and so $x(t) = \exp \left( \int_0^t \lambda(s) ds \right)$ is defined for all $t \in \R_{\ge 0}$.

Now, suppose that $\HH$ is not acyclic, i.e. $\HH$ has a sub-hyperchain $\HH'$ which is a directed cycle. It suffices to show that $\HH'$ has a finite-time blow-up. Label the species in $\HH'$ to be $X_1, X_2, \ldots, X_n$, where $X_i \dasharrow X_{i+1} \in \HH'$, and the summation over index $i$ is mod $n$.  
By Proposition \ref{prop:nondim}, we may assume without loss of generality that all rate constants in $\HH'$ are 1. 
Let $x(0) = (x_1(0), \ldots, x_n(0)) \in \R^n_{> 0}$ be an initial value. Suppose that the initial value problem has a solution for all $t \in [0, \infty)$. Let $x(t) = (x_1(t), \ldots, x_n(t))$ be one of these solutions. 
By Corollary \ref{cor:int-ln}, $\{t : x_{i+1}(t) \le x_i(t)\}$ has finite measure for all $i \in \{1, \ldots, n\}$. But then, it follows that $\{t : x_1(t) \le x_1(t)\}$ has finite measure, which is obviously a contradiction. Therefore, the assumption that the initial value problem has a solution for all $t \in [0, \infty)$ must be false. 
\end{proof}

\section{Relative Concentration System of a Hyperchain} \label{sec:relativeconcentration}

Every reaction in a hyperchain increases the concentration of some species. 
In a realistic set-up, the excess concentration might diffuse out of the reaction volume and into the environment, thus keeping the relevant system variables finite for all time. To model this, we might augment the hyperchain system with appropriate outflow reactions of the type $X_i \to 0$. This design hews close to the experimental setup known as {\em chemostat}, a thermodynamically open system that allows for mass transfer through inflows and outflows. 

Alternatively, we perform a nonlinear transformation on the coordinates and study the dynamics of the relative concentrations of species. The {\em relative concentration of $X_i$} (i.e. relative to the total concentration) is defined to be $x_i/\bar x$ where $\bar x := \sum_{j=1}^n x_j$. After an appropriate uniform rescaling of the reaction rates by the total concentration, the resulting dynamical system of relative concentrations of $X_i$ is confined within the bounded region $S_{n} = \left \{(x_1, \ldots, x_n) \in \R^n | 
x_i \ge 0, \sum_{i=1}^n x_i = 1 \right \}$. It turns out that the two apparently distinct solutions for the problem of ever-increasing concentrations are, in fact, intimately related. We give a brief account here, see also \cite{hofbauer1988theory}.

Let $S_{n} = \left \{(x_1, \ldots, x_n) \in \R^n | 
x_i \ge 0, \sum_{i=1}^n x_i = 1 \right \}$ be the standard simplex in the nonnegative orthant of $n$-dimensional Euclidean space. We define a nonlinear, invertible change of variables:
\begin{align} \label{eq:nonlin}
 \Phi: \R^n_{\ge 0} \setminus \{0 \} &\to \R^n_{\ge 0} \setminus \{0 \} = {\color{teal!100} S_{n}} \times {\color{magenta} \R_{> 0}} \nonumber\\
 (x_1, \ldots, x_n) &\mapsto (\widetilde x_1, \ldots, \wt x_n, \ol x) \\
\text{\color{magenta} (total conc) } \ol x &= \sum x_i \nonumber \\
\text{\color{teal!100} (relative conc) } \widetilde x_i &= x_i \left /\ol x \quad (1\le i \le n) \right. \nonumber
\end{align}
where we refer to $\ol x$ as the total concentration and to $\wt x_i$ as the relative concentration of the species $X_i$. To make the distinction sharper, we will refer to $x_i$ as the absolute concentration of the species $X_i$. 
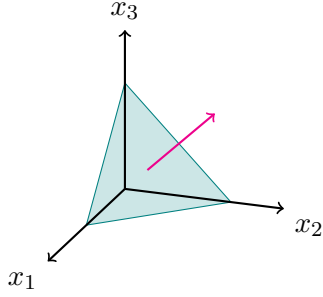
\begin{figure}
\begin{center}
\tdplotsetmaincoords{70}{110}
\begin{tikzpicture}[scale=1.5,tdplot_main_coords]
    \def\x{.5}
    \filldraw[
        draw=teal,%
        fill=teal!20,%
    ]          (1,0,0)
            -- (0,1,0)
            -- (0,0,1)
            -- cycle;
    \draw[thick,->] (0,0,0) -- (2,0,0) node[anchor=north east]{$x_1$};
    \draw[thick,->] (0,0,0) -- (0,1.5,0) node[anchor=north west]{$x_2$};
    \draw[thick,->] (0,0,0) -- (0,0,1.5) node[anchor=south]{$x_3$};
     \draw[magenta,thick,->] (1/3,1/3,1/3) -- (4/3,4/3,4/3) node[anchor=south]{};
\end{tikzpicture}
\caption{Change of variables to $S_n \times \R_{> 0}$}\label{fig:change}
\end{center}
\end{figure}
The variables $\wt x_i$ are restricted to the simplex $S_n$, while $\ol x$ is along a ray orthogonal to the simplex, starting at the origin and heading to a point at infinity. Figure \ref{fig:change} depicts $S_3$ (in teal) and the corresponding ray (in magenta). Consider the system of ODEs $\dot x = \rho(x)$ where $x, \rho(x) \in \R^n_{\ge 0} \setminus \{0\}$ with the property that $\rho$ is a homogeneous polynomial system with every term a monomial of degree $m$. Then, the automorphism $\Phi$ transforms $\dot x = \rho(x)$ into
\begin{align} \label{eq:rel1}
\dot{\wt x_i} &= \ol x^{m-1} \left( \rho_i (\wt x) - \wt x_i \ol \rho(\wt x) \right) \nonumber\\
\dot{\ol x} &= \ol x^{m} \ol \rho(\wt x), 
\end{align} 
where $\ol \rho := \sum_i \rho_i$.
For the hyperchain system $\rho_i(x) = x_i f_i(x)$, and so the above is equivalent to
\begin{align} \label{eq:rel3}
\dot{\wt x_i} &=  \ol x^{m-1} \wt x_i \left( f_i( \wt x) - \ol \rho (\wt x) \right) \nonumber\\ 
\dot{\ol x} &= \ol x^m ~\ol \rho (\wt x) 
\end{align}
where $\ol \rho (\wt x) = \sum_{i,j} k_{ij} \wt{x_i} \wt{x_j}$. 
After a time-rescaling (see for instance \cite{kuehn2015multiple}), we get the following system which has the same set of orbits as \eqref{eq:rel1}
\begin{align} \label{eq:rel2}
\wt x_i' := \dv{\wt x_i}{\tau} &=  \left( \rho_i (\wt x) - \wt x_i \ol \rho(\wt x) \right) \nonumber \\
\ol x' := \dv{\ol x}{\tau} &= \ol x ~ \ol \rho(\wt x). 
\end{align} 
The absence of $\ol x$ in $x_i'$ shows that the relative concentration system on $S_n$ decouples from the total concentration $\ol x$. 

Now, we consider the alternative of embedding the dynamical system of a hyperchain in a chemostat. Formally, we augment the underlying autocatalytic reactions with the set of flow reactions $\{ X_i \xrightarrow{\ell} 0 | 0 \le i \le n \}$, one for each species in the hyperchain. Each species flows out at the rate $\ell$. Including the outflow reactions results in a modification of the system $\dot x = \rho(x)$ to $\dot x = \rho(x) - \ell x$. 
The total concentration is governed by $\dot{\ol x} = \ol \rho(x) - \ell \ol x$, which is held at a constant value if $\ell = \ol \rho(x) / \ol x$. With this assumption, the dynamical system is
\[
\dot x = \mu(x) := \rho(x) - \wt x \ol \rho(x)
\]
If we apply $\Phi$ and rescale time $\tau := \ol x^{m-1} t$ as in \eqref{eq:rel2}, we get 
\begin{align*} 
\wt x_i' := \dv{\wt x_i}{\tau} &=  \left( \mu_i (\wt x) - \wt x_i \ol \mu(\wt x) \right) \nonumber \\
\ol x' := \dv{\ol x}{\tau} &= \ol x ~ \ol \mu(\wt x). 
\end{align*} 
It is easy to check that $\ol \mu(\wt x) =0$ and so $\wt x ' = \mu(\wt x) = \rho(\wt x) - \wt x \ol \rho (\wt x)$.  The relative concentration system is identical to \eqref{eq:rel2} while the total concentration is held to a constant value of 1 instead of going to infinity. Thus the chemostat version, which is in principle, realizable in a laboratory, has the virtue that the dynamics are truly restricted to $S_n$.

\section{Dynamics of Relative Concentrations in a Hyperchain} \label{sec:dynamics_relconc}

Our goal in this article is to relate the network topology of a hyperchain $\HH$ or a labeled hyperchain $(\HH,K)$ with its dynamical properties under mass-action kinetics. As a matter of convenience in phrasing, we abbreviate ``mass action ODE system generated by $(\HH,K)$'' to the ``hyperchain system $(\HH,K)$''. 
For instance, we say ``$(\HH,K)$ has dynamical property $\PP$'' as an abbreviation for the more accurate ``the mass action ODE system generated by $(\HH,K)$ has dynamical property $\PP$''. 
Similarly, ``$\HH$ has dynamical property $\PP$'' is an abbreviation for ``for any choice of mass action kinetics $K$, the mass action ODE system generated by $(\HH,K)$ has dynamical property $\PP$''. 

From this section onwards, we are {\em only} concerned with the dynamics of {\em relative} concentrations of $X_i$ under mass action kinetics. We drop the $~\widetilde{  }~$ and denote the relative concentration of the species $X_i$ by $x_i$. 
For $x =(x_1, \ldots, x_n)^T \in S_n$, hyperchain dynamics under mass action kinetics are governed by 
\begin{equation} \label{eq:comp}
\dot x_i  = x_i \left( \sum_{j=1}^n k_{ji} x_j - \ol\rho(x) \right) 
\end{equation}
where $\ol \rho(x) = \sum_{\ell, j}  k_{j\ell} x_j x_\ell$. 
We state here two alternative presentations of \eqref{eq:comp}. Let $f_i(x) = \sum_{j=1}^n k_{ji} x_j$, $f(x) = (f_1(x), \ldots, f_n(x))^T$, $K$ be the rate constant matrix, and $\one := (1, \ldots, 1)^T \in \R^n$. Then $f(x) = K^T x$, i.e. $\sum_{j=1}^n k_{ji} x_j = \left(K^T x\right)_i$ and $\bar \rho(x) =  x^T K^T x = x^T K x$. Moreover, \eqref{eq:comp} is equivalent each of the following:
\begin{align}
\dot x &= x \ast \left( f(x) - (x \cdot f) \one \right)   \label{eq:vec} \\
\dot x &= x \ast \left( K^T x - \left( x^T K^T x \right) \one \right) = x \ast K^T x - \left( x^T K^T x \right) x  \label{eq:mat}
\end{align}

Since the fundamental object of study, the hyperchain is a (directed) graph, we begin with some graph-theoretic concepts. 
 
 \begin{definition}
 Consider a directed graph (digraph) $G = (V,D)$ where $V$ is the set of vertices and $D \subseteq V \times V$ is the set of directed edges. We will denote a directed edge $(u,v)$ by $u \dashrightarrow v$, and refer to this edge as exiting vertex $u$ and entering vertex $v$. 
 \been
 \item The {\em indegree (outdegree)} of a vertex $v \in V$ is the number of edges entering (exiting) $v$. An {\em initial (terminal) node} is a vertex with indegree (outdegree) 0. $G$ is a {\em rooted graph} if $G$ has at least one initial node, and {\em unrooted} otherwise.
 \item A {\em linear graph} is a digraph where every vertex has both indegree and outdegree equal to 1. 
 \item A {\em subgraph of $G = (V,D)$} is a graph $G'=(V',D')$ such that $V' \subseteq V$, $D' \subseteq D$, and  $u \dashrightarrow v \in D'$ implies that $\{u,v\} \subseteq V'$. $G'$ is a {\em spanning subgraph of $G$} if $V' = V$. 
 \item A {\em strongly connected graph} is a digraph for which there exists a directed path of edges between any two vertices. 
 \item A {\em cycle} is a strongly connected, linear graph. An {\em even cycle} is a cycle with an even number of edges. A linear digraph is said to be {\em odd (even)} if it has odd (even) number of even cycles. 
 \item A {\em Hamiltonian cycle} of $G = (V,D)$ is a strongly connected, spanning, linear subgraph of $G$. $G$ is {\em Hamiltonian} if $G$ has a Hamiltonian cycle. 
 \enen
 \end{definition}
 
\begin{proposition}
Suppose that $\HH$ is a linear graph. Then, for any choice of mass action rate constants $K$ and  $K'$, the systems $(\HH,K)$ and $(\HH,K')$ are diagonally conjugate. 
\end{proposition}
\begin{proof}
Consider two hyperchain systems $(\HH, K)$ and $(\HH, K')$. Then by Lemma \ref{prop:nondim}, the two systems are diagonally conjugate to $(\HH, K_0)$ and $(\HH, K_0')$, respectively, where $K_0$ and $K_0'$ are defined as in the proof of Lemma \ref{prop:nondim}. However, from the definition of a linear network, it follows that $K_0 = K_0' = \one$, where $\one$ is the assignment where every reaction rate constant is equal to 1. It follows that $(\HH, K)$ and $(\HH, K')$ are diagonally conjugate. 
\end{proof}

We devote the rest of this section to finding all equilibria of a hyperchain system, and the stability of these equilibria. 


\subsection{Positive equilibria} \label{sec:poseq}

From \eqref{eq:vec}, the positive equilibria $x \in S_n$ of a hyperchain system are solutions of 
\begin{equation} \label{eq:poseq}
f_1(x) = f_2(x) = \ldots = f_n(x) = \ol \rho(x).
\end{equation}
So $x$ is a positive equilibrium of a hyperchain system if and only if $f = (x \cdot f) \one$ and $x \cdot \one =1$. We show that the final condition is redundant. 
\begin{lemma} \label{lem1:invset}
Let $f = c \one$ for some $c \ne 0$. Then $c = x \cdot f$ if and only if $x \cdot \one =1$. 
\end{lemma}
\begin{proof}
Suppose that $c = x \cdot f$. Then $f = (x \cdot f) \one$ and so $x \cdot f = (x \cdot f) x \cdot \one$. Since $c = x \cdot f \ne 0$, $x \cdot \one =1$. Suppose that $x \cdot \one = 1$. Then $1 = x \cdot \left( \frac{1}{c}\right) f$ and so $c = x \cdot f$. 
\end{proof}
Thus the positive equilibria of a hyperchain system are solutions of $f = (x \cdot f) \one$ where $f = K^T x$. In matrix notation, the positive equilibria are solutions of the system
\begin{align} \label{eq:mateq}
K^T x =  \left( x^T K^T x \right) \one
\end{align}
Procedurally, it is more convenient to find positive solutions of $K^T z = \one$ and then find the positive equilibria from $x = z/\ol z$. 
We now give a necessary and sufficient condition for the existence of a positive equilibrium. 

\begin{theorem}
There is a $K$ such that $(\HH, K)$ has a positive equilibrium if and only if $\HH$ is unrooted.
\end{theorem}
\begin{proof} 
The positive equilibria of $(\HH,K)$ are positive solutions of \eqref{eq:poseq} where $f_i(x) = \sum_{j=1}^n k_{ji} x_j$ for $1 \le i \le n$. 

Suppose that $\HH$ is rooted. Let $x^* \in \R^n_{>0}$ be a solution of \eqref{eq:poseq}. If $X_j$ is an initial node, then $f_j(x^*) = 0$. But this means that $f_i(x^*) = 0$ for all $i \in \{1, \ldots, n\}$. Thus the component of $x^*$ for every non-terminal node must be $0$, and there is at least one non-terminal node in $\HH$, which contradicts the positivity of $x^*$. Therefore, \eqref{eq:poseq} does not have a positive solution.

Suppose that $\HH$ is unrooted. Let $d_{in}(X)$ be the indegree of the node $X$. Since $\HH$ is unrooted, $d_{in}(X) \ge 1$ for every node $X$. Define $K$ as follows. For every node $X_i$, and every edge $X_j \dashrightarrow {X_i} \in \HH$,  let $k_{ji} = 1/d_{in}(X_i)$.  Then
\[
f_i\left(\frac{1}{n} \one\right) = \frac{1}{d_{in}(X_i)} \frac{1}{n} \left( \sum_{X_j \dashrightarrow X_i \in \HH} 1 \right) =  \frac{1}{n} 
\]
Therefore $\frac{1}{n} \one = \left(\frac{1}{n},\ldots,\frac{1}{n}\right)$ is a positive equilibrium of  $(\HH,K)$. 
\end{proof}

Reaction networks generate nonlinear polynomial dynamical systems under mass-action kinetics, and so are generally multistationary, i.e. they admit more than one positive equilibrium  \cite{craciun2005multiple,joshi2015survey}. 
The number of equilibria of a hyperchain system is severely restricted. In particular, a hyperchain system cannot be multistationary. 

\begin{theorem}
Let $(\HH,K)$ be a hyperchain system. $(\HH,K)$ can have either $0, 1$ or infinitely many positive equilibria, and there are no other possibilities. 
\end{theorem}
\begin{proof}
The positive equilibria of $(\HH,K)$ are positive scalar multiples of $z$, where $z$ is a positive solution of $K^T z = \one$. Since this is a linear equation, there are only three possibilities: $0, 1$ or infinitely many solutions. 
\end{proof}

We now give a necessary and sufficient condition for the existence of a {\em unique} positive equilibrium. 

\begin{theorem} \label{thm:uniqueposeq}
Consider a hyperchain $\HH$.
The following are equivalent.
\been
\item $\HH$ has a  spanning linear subgraph.
\item There is a $K$ such that $(\HH,K)$ has a unique positive equilibrium. 
\item There is a $K$ such that $(\HH,K)$ has an isolated positive equilibrium.
\item There is a $K$ such that $\det(K) \ne 0$. 
\enen
\end{theorem}
\begin{proof}
Suppose that there is a $K$ such that $(\HH,K)$ has a unique positive equilibrium. Then clearly $(\HH,K)$ has an isolated positive equilibrium and the matrix $K$ is invertible, which implies that $\det(K) \ne 0$. If $\det(K) \ne 0$, then in particular there must be a nonzero term in the expansion of the determinant. Each nonzero term arises from a  spanning linear subgraph by classical results in theory of directed graphs (see \cite{harary1959graph,harary1962determinant}) and so $\HH$ has a  spanning linear subgraph. Suppose that $\HH$ has a  spanning linear subgraph, denote it by $\HH'$ and consider ${K_\ep}$ where all the rate constants along $\HH'$ are $1$, while all the rate constants along $\HH \setminus \HH'$ are $\ep$. Then $(\HH',K) = (\HH,K_0)$ is a union of disjoint cycles such that each vertex is in exactly one cycle, and each rate constant is 1. It is easy to check that $\det\left(K_0\right)=1$ and $(1/n,\ldots,1/n)^T$ is an eigenvector of $K_0^T$ . Since $\sum_{i=1}^n 1/n =1$, by Lemma \ref{lem1:invset}, $(1/n,\ldots,1/n)^T$ is the unique positive equilibrium of $(\HH,K_0)$. For sufficiently small $\ep$, $\det\left(K_\ep\right) >0$ and \eqref{eq:mateq} has a unique positive solution. This completes the proof. 
\end{proof}
\begin{remark}

Note that $\det(K) \ne 0$ does not guarantee that $(\HH,K)$ has a positive equilibrium because the unique solution to \eqref{eq:mateq} need not be positive. However, by Theorem \ref{thm:uniqueposeq}, the condition does guarantee that there exists a $K'$ such that  $(\HH,K')$ has a positive equilibrium. 
\end{remark}

\begin{example}
Consider the following $(\HH,K)$, where all unlabelled edges are assumed to have rate constant 1. Consider two  spanning linear subgraphs of $\HH$ (the one in the middle is depicted in red and the one on the right is depicted in brown), which we refer to as $\HH_1$ and $\HH_2$, respectively. $\HH_1$ has 0 even cycles, and so $\HH_1$ is even, while $\HH_2$ has 1 even cycle, and so $\HH_2$ is odd. 
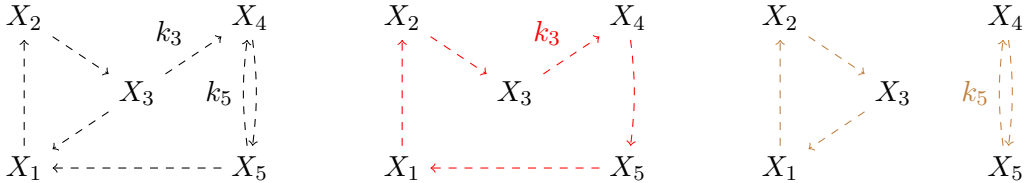
\begin{figure}[h!]
\begin{center}
  \begin{tikzpicture}[baseline={(current bounding box.center)}, scale=0.5]
   \node[state] (1)  at (-3,-2)  {$X_1$};
   \node[state] (2)  at (-3,2)  {$X_2$};
   \node[state] (3)  at (0,0)  {$X_3$};
   \node[state] (4)  at (3,2)  {$X_4$};
   \node[state] (5)  at (3,-2)  {$X_5$};

   \path[->,dashed]
    (3) edge[] node {} (1)
     (1) edge[] node {} (2)
    (2) edge[] node {} (3)
    (3) edge[] node {$k_3$} (4)
    (4) edge[bend left] node {} (5)
    (5) edge[bend left] node {$k_5$} (4)
    (5) edge[] node {} (1)
;
  \end{tikzpicture}
    \quad \quad \quad 
   \begin{tikzpicture}[baseline={(current bounding box.center)}, scale=0.5]
   \node[state] (1)  at (-3,-2)  {$X_1$};
   \node[state] (2)  at (-3,2)  {$X_2$};
   \node[state] (3)  at (0,0)  {$X_3$};
   \node[state] (4)  at (3,2)  {$X_4$};
   \node[state] (5)  at (3,-2)  {$X_5$};

   \path[->,dashed]
     (1) edge[red] node {} (2)
    (2) edge[red] node {} (3)
    (3) edge[red] node {$k_3$} (4)
    (4) edge[red,bend left] node {} (5)
    (5) edge[red] node {} (1)
;
  \end{tikzpicture}
  \quad \quad \quad 
\begin{tikzpicture}[baseline={(current bounding box.center)}, scale=0.5]
   \node[state] (1)  at (-3,-2)  {$X_1$};
   \node[state] (2)  at (-3,2)  {$X_2$};
   \node[state] (3)  at (0,0)  {$X_3$};
   \node[state] (4)  at (3,2)  {$X_4$};
   \node[state] (5)  at (3,-2)  {$X_5$};

   \path[->,dashed]
    (3) edge[brown] node {} (1)
     (1) edge[brown] node {} (2)
    (2) edge[brown] node {} (3)
    (4) edge[brown,bend left] node {} (5)
    (5) edge[brown,bend left] node {$k_5$} (4)
;
  \end{tikzpicture}
  \caption{On the left is $(\HH,K)$, where the unlabeled edges are assumed to have assigned a rate constant of 1. On the middle and the right are the two  spanning linear subgraphs of $\HH$, which we refer to as $\HH_1$ and $\HH_2$, respectively.}
  \end{center}
 \end{figure}
From results in \cite{harary1962determinant}, it follows that $\det(K |_{\HH_1})=k_3$ and $\det(K |_{\HH_2})=-k_5$, and so $\det(K)=k_3-k_5$. If $k_3 \ne k_5$, then $K$ is invertible and so there exists at most one positive equilibrium. In fact, we can explicitly solve \eqref{eq:mateq}, 
\[
\frac{1}{\bar \rho(x)} x = \left(1,1,\frac{k_5-1}{k_5-k_3},1,\frac{1-k_3}{k_5-k_3}\right)
\]
So, in fact, a unique positive equilibrium exists if and only if either $k_3<1<k_5$ or $k_5<1<k_3$. 
If $k_3=k_5$, the positive equilibria are solutions of $x_3+x_5=x_1=x_2=k_3(x_3+x_5)=x_4$ along with $\sum x_i =1$. If $k_3=k_5 \ne 1$, then there are no positive equilibria. If $k_3=k_5 = 1$, there is a  continuum of solutions parametrized by $b \in [0,1/4]$, 
\[
x= \left( \frac{1}{4}, \frac{1}{4},b, \frac{1}{4},\frac{1}{4} -b\right) \quad \mbox{where  ~~} 0 \le b \le \frac{1}{4}. 
\]
\end{example}

\begin{theorem} \label{thm:maxone}
If a hyperchain $\HH$ has at least one  spanning linear subgraph and if all its  spanning linear subgraphs have the same parities (either all even or all odd), then $(\HH,K)$ has at most one positive equilibrium for any $K>0$. 
\end{theorem} 
\begin{proof}
Since $\HH$ has a  spanning linear subgraph, $\det(K)$ has at least one nonzero term, and since all  spanning linear subgraphs have the same parities, all nonzero terms of $\det(K)$ have the same parities (either all negative or positive). Both of these conclusions follow from Theorem 1 of \cite{harary1962determinant}. This implies that $\det(K)\ne 0$ and so \eqref{eq:mateq} has at most one solution, from which the result follows.  
\end{proof}

\begin{corollary}
If a hyperchain $\HH$ has a unique  spanning linear subgraph, then $(\HH,K)$ has at most one positive equilibrium for any $K>0$. 
\end{corollary} 
\begin{proof}
This is a simple consequence of Theorem \ref{thm:maxone}
\end{proof}
\begin{example}
Consider the following three network types. Each wavy arrow represents some sequence of edges with an arbitrary number of intermediate species. 
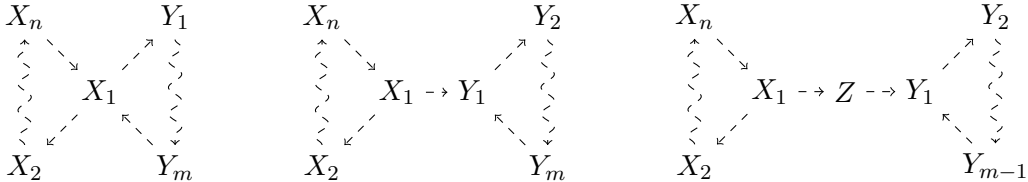
\begin{equation*}
\begin{tikzpicture}[baseline={(current bounding box.center)}, scale=0.5]
   \node[state] (x1)  at (0,0)  {$X_1$};
   \node[state] (y1)  at (2,2)  {$Y_1$};
      \node[state] (ym)  at (2,-2)  {$Y_m$};
      \node[state] (x2)  at (-2,-2)  {$X_2$};
       \node[state] (xn)  at (-2,2)  {$X_n$};
            

   \path[->,dashed]
    (x1) edge[] node {} (y1)

    (ym) edge[] node {} (x1)
    (x1) edge[] node {} (x2)

    (xn) edge[] node {} (x1)

     (y1) edge[dashed,snake=snake] node {} (ym)
      (x2) edge[dashed,snake=snake] node {} (xn)
;
  \end{tikzpicture}
    \quad \quad \quad
  \begin{tikzpicture}[baseline={(current bounding box.center)}, scale=0.5]
   \node[state] (x1)  at (0,0)  {$X_1$};
   \node[state] (y1)  at (2,0)  {$Y_1$};
   \node[state] (y2)  at (4,2)  {$Y_2$};
      \node[state] (ym)  at (4,-2)  {$Y_m$};
      \node[state] (x2)  at (-2,-2)  {$X_2$};
       \node[state] (xn)  at (-2,2)  {$X_n$};
            

   \path[->,dashed]
    (x1) edge[] node {} (y1)
    (y1) edge[] node {} (y2)

    (ym) edge[] node {} (y1)
    (x1) edge[] node {} (x2)

    (xn) edge[] node {} (x1)

     (y2) edge[dashed,snake=snake] node {} (ym)
      (x2) edge[dashed,snake=snake] node {} (xn)
;
  \end{tikzpicture}
  \quad \quad \quad
\begin{tikzpicture}[baseline={(current bounding box.center)}, scale=0.5]
   \node[state] (x1)  at (0,0)  {$X_1$};
   \node[state] (y1)  at (4,0)  {$Y_1$};
   \node[state] (y2)  at (6,2)  {$Y_2$};
      \node[state] (ym)  at (6,-2)  {$Y_{m-1}$};
      \node[state] (x2)  at (-2,-2)  {$X_2$};
       \node[state] (xn)  at (-2,2)  {$X_n$};
        \node[state] (z)  at (2,0)  {$Z$};
            

   \path[->,dashed]
    (x1) edge[] node {} (z)
    (z) edge[] node {} (y1)
    (y1) edge[] node {} (y2)

    (ym) edge[] node {} (y1)
    (x1) edge[] node {} (x2)

    (xn) edge[] node {} (x1)

     (y2) edge[dashed,snake=snake] node {} (ym)
      (x2) edge[dashed,snake=snake] node {} (xn)
;
  \end{tikzpicture}
\end{equation*}
Only the middle network type has a  spanning linear subgraph and so only for this network type there is a choice of $K$ such that the resulting mass action system has a unique positive equilibrium. Furthermore, the middle network type has a unique  spanning linear subgraph, and so for every $K$, the mass action system has at most one positive equilibrium. 
\end{example}

\subsection{Boundary equilibria}
To find the boundary equilibria of $(\HH,K)$, we simply remove the nodes whose concentration is zero along with all its adjoining edges, and then consider the positive equilibria of the remaining network. We now make this idea precise. 
\begin{lemma} \label{lem:bdry}
Consider a hyperchain system $(\HH,K)$. Let $\BB = \{x_{i_1} = 0, \ldots, x_{i_\kappa} = 0\}$ be a boundary of $S_n$. Then $\BB$ is invariant under \eqref{eq:comp}. 
\end{lemma}
\begin{proof}
It is clear from \eqref{eq:comp} that if $x_i=0$ then $\dot x_i =0$. 
\end{proof}
Consider a (labeled or unlabeled) graph $G=(V,E)$ and let $V' \subseteq V$. The graph {\em induced} by $V'$ is $G[V'] = (V',E')$ where $E' =\{(\cdot, v), (v,\cdot) | v \in V'\}$. Let $\II$ be a nonempty subset of $\{1, \ldots, n\}$. If $\HH$ is a hyperchain on the $n$ species, $\{X_1, \ldots, X_n\}$, then define the subset of species $X_\II = \{X_i : i \in \II \}$. Clearly, the induced graph $\HH[X_{\II}]$ is a hyperchain, and the induced labeled graph $(\HH,K)[X_\II]$ is a hyperchain system.

\begin{lemma} \label{lem:notational}
Consider a hyperchain system $(\HH,K)$ with the associated dynamical system $\dot x = g(x) = (g_1(x), \ldots, g_n(x))$. Let $\II$ be a nonempty subset of $\{1, \ldots, n\}$ and $\II^c$ its complement. Consider a boundary of $S_n$ defined by $\BB_\II := \{x ~|~ x_i = 0  \iff i \in \II\}$. Denote the dynamical system associated with $(\HH,K)[X_{\II^c}]$ by $\dot {\hat x} = \hat g(\hat x)$. Suppose that $\hat x=(\hat x_i: i \in \II^c)$ is in the state space of $(\HH,K)[X_{\II^c}]$. We define $x = (x_1, \ldots, x_n)$ in the state space of $(\HH,K)$ as follows
\[
x_i = \begin{cases} \hat x_i &\quad \quad \mbox{ if } i \in \II^c \\ 
0 &\quad \quad \mbox{ if } i \in \II
\end{cases}
\]
Then, $\hat x$ is a positive point in the state space of $\left(\HH,K\right)[X_{\II^c}]$ if and only if $x \in \BB_\II$. 
\end{lemma}
\begin{proof}
The proof is immediate. 
\end{proof}

\begin{theorem}
Consider a hyperchain system $(\HH,K)$ and let $\II$ be a nonempty subset of $\{1, \ldots, n\}$. In the notation of Lemma \ref{lem:notational}, $\hat x$ is a positive equilibrium of $\left(\HH,K\right)[X_{\II^c}]$ if and only if $x$ is an equilibrium of $(\HH,K)$ in $\BB_\II$. 
\end{theorem}
\begin{proof}
By Lemma \ref{lem:bdry}, $\BB_\II$ is invariant for $(\HH,K)$, i.e. for $i \in \II$,  $x_i = g_i(x) = 0$. For any $i \in \II^c$ and $x \in \BB_\II$,  
\begin{align*}
g_i(x)  &= x_i \left( \sum_{j: X_j \dasharrow X_i \in \HH} k_{ji} x_j - \sum_{i,j: X_j \dasharrow X_i \in \HH} k_{ji}   x_i  x_j \right) \\
&=   x_i \left( \sum_{j: X_j \dasharrow X_i \in \HH[X_{\II^c}]} k_{ji}   x_j - \sum_{i,j: X_j \dasharrow X_i \in \HH[X_{\II^c}]} k_{ji}   x_i  x_j \right) \\
&=   \hat x_i \left( \sum_{j: X_j \dasharrow X_i \in \HH[X_{\II^c}]} k_{ji}   \hat x_j - \sum_{i,j: X_j \dasharrow X_i \in \HH[X_{\II^c}]} k_{ji}   \hat x_i  \hat x_j \right) = \hat g_i(\hat x). 
\end{align*}
The result immediately follows. 
\end{proof}

\subsection{Stability of Positive Equilibria}

Let $\dot x = g(x) = x \ast \left( f(x) - (x \cdot f) \one \right)$ be the dynamical system for $(\HH,K)$. 
We saw in Section \ref{sec:poseq} that a unique positive equilibrium exists if and only if $K$ is invertible and the unique solution of the equation $K^T z = \one$ is positive. We wish to establish the conditions under which the unique positive equilibrium (when it exists) is stable. 

For $z \in \R^n$, by $\diag(z)$ we mean the unique diagonal matrix whose diagonal entries form the vector $z$. 
We now state the main theorem of this section, and devote the rest of the section to proving it along with a few ancillary results.

\begin{theorem} \label{thm:main_stab}
A hyperchain system $(\HH,K)$ on $n$ species has a unique, positive linearly stable equilibrium if and only if (i) $K$ is invertible, (ii) $(K^T)^{-1} \one \in \R^n_{> 0}$, and (iii) $\diag\left((K^T)^{-1} \one \right) K^T$ has exactly $n-1$ eigenvalues with negative real part. 
\end{theorem}
The proof is deferred to the end of the section.

\begin{proposition}
For $x \in S_n$, consider the dynamical system,  
$\dot x = g(x) = x \ast \left( K^T x - \left( x^T K^T x \right) \one \right)$. The Jacobian matrix of $g(x)$ at $z \in \R^n$ is 
\begin{equation} \label{eq:jac1}
D_x g|_{x=z} = \diag\left(K^T z - (z^T K z) \one\right) + \diag(z) K^T - z z^T \left(K + K^T\right). 
\end{equation}
\end{proposition}
\begin{proof}
In components, 
$
g_i(x) = x_i \left( f_i(x) - \ol\rho (x)\right)
$
where $f_i(x) := \sum_{j=1}^n k_{ji} x_j$ and $\bar \rho(x) := x \cdot f(x)$. Differentiating with respect to $x_j$, 
\begin{align*}
\pdv{g_i(x)}{x_j} = \delta_{ij} (f_i(x) - \bar \rho (x)) + x_i \left( \pdv{f_i(x)}{x_j} - \pdv{\bar \rho(x)}{x_j}  \right)
\end{align*}
where 
\[
\pdv{f_i(x)}{x_j} = \pdv{}{x_j} \sum_{m=1}^n k_{mi} x_m = \sum_{m=1}^n k_{mi} \delta_{mj} = k_{ji} 
\]
and 
\[
\pdv{\bar \rho(x)}{x_j} = \pdv{}{x_j} \sum_{i=1}^n \sum_{m=1}^n k_{mi} x_m x_i = \sum_{i=1}^n \left( k_{ij} + k_{ji} \right) x_i. 
\]
Putting these together, we get the $(i,j)$th entry of the Jacobian matrix, 
\begin{equation} \label{eq:jac_comp}
\pdv{g_i(x)}{x_j} = \delta_{ij} (f_i(x) - \bar \rho (x)) + k_{ji} x_i - x_i \sum_{m=1}^n \left( k_{mj} + k_{jm} \right) x_m. 
\end{equation}
In matrix notation
\begin{align*}
D_x g|_{x=z} &= \diag\left(f(z) - \bar \rho (z)\one\right) + \diag(z) K^T - z z^T \left(K + K^T\right), \\
 &= \diag\left(K^T z - (z^T K z) \one\right) + \diag(z) K^T - z z^T \left(K + K^T\right), 
\end{align*}
as desired.
\end{proof}

\begin{corollary}
For $x \in S_n$, consider the dynamical system,  
$\dot x = g(x) = x \ast \left( K^T x - \left( x^T K^T x \right) \one \right)$. The Jacobian matrix of $g(x)$ at a positive equilibrium $z \in \R^n_{>0}$ is given by 
\begin{equation} 
D_x g|_{x=z} = \diag(z) K^T - z z^T \left(K + K^T\right). 
\end{equation}
\end{corollary}
\begin{proof}
Positive equilibria are solutions of $K^T x = \left(x^T K x\right) \one$, 
and so at a positive equilibrium $z$, the first term in \eqref{eq:jac1} is 0. 
\end{proof}
The following result is Theorem 2.1 in Ding \& Zhou \cite{ding2007eigenvalues}. 
\begin{lemma} \label{lem:ding}
Suppose that $M \in \C^{n \times n}$ has eigenvalues $(\lambda_1, \ldots, \lambda_n)$. Let $u \in \C^n$ be an eigenvector of $M$ with the associated eigenvalue $\lambda_1$, and let $v \in \C^n$. Then (i) the eigenvalues of $M+ uv^T$ are $(\lambda_1 + v^Tu, \lambda_2, \ldots, \lambda_n)$, and (ii) $u$ is an eigenvector of $M+ uv^T$ corresponding to the eigenvalue $\lambda_1 + v^Tu$. 
\end{lemma}
\begin{proof}
The first part is proved in Theorem 2.1 of Ding \& Zhou. For the second part, $(M + uv^T) u = Mu + u (v^Tu) = (\lambda_1 + v^Tu) u$. 
\end{proof}

\begin{lemma} \label{lem:ev}
If $z$ is a positive equilibrium of $(\HH,K)$, then $z$ is an eigenvector of the matrix $M = \diag(z) K^T$ with the associated eigenvalue $z^TK z$. 
\end{lemma}
\begin{proof}
We use the positive equilibrium condition $K^T z = \left(z^T K z\right) \one$,
\begin{align*}
Mz & = \left(\diag(z) K^T\right) z =  \diag(z) \left(K^T z\right) =  \diag(z) \left(\left(z^T K z\right) \one\right) \\
& =  \left(z^T K z\right) \diag(z)  \one =   \left(z^T K z\right) z,
\end{align*}
which completes the proof.
\end{proof}

\begin{theorem} \label{thm:eigD}
Let $z$ be a positive equilibrium of $(\HH,K)$. Let the eigenvalues of $\diag(z) K^T$ be $\left( \lambda_1 := z^T K z, \lambda_2, \ldots, \lambda_n \right)$. Then the following hold:
\been[(i)]
\item  the eigenvalues of $D_x g(z)$ are $\left( -\lambda_1, \lambda_2, \ldots, \lambda_n \right)$, and 
\item $z$ is an eigenvector of $D_x g(z)$ corresponding to the eigenvalue $-\lambda_1 = - z^T K z$. 
\enen
\end{theorem}
\begin{proof}
At a positive equilibrium $z$, $D_x g(z) = \diag(z) K^T - z z^T \left(K + K^T\right) = M + u v^T$, where $M = \diag(z) K^T$, $u=-z$, and $v^T = z^T \left(K + K^T\right)$. By Lemma \ref{lem:ev}, $u=-z$ is an eigenvector of $M$ with the associated eigenvalue $\lambda_1$. By Lemma \ref{lem:ding}, the eigenvalues of $D_xg(z)$ are $\left( \lambda_1+ v^Tu, \lambda_2, \ldots, \lambda_n \right)$, and $z$ is an eigenvector of $D_xg(z)$ with the corresponding eigenvalue $\lambda_1+ v^Tu$. 
 It is only required to show that $\lambda_1 + v^T u = -\lambda_1$. Indeed,
\begin{equation*}
\lambda_1 + v^T u = z^T K z - z^T \left(K + K^T\right) z = z^T K z - 2 z^T K z = - z^T K z = -\lambda_1, 
\end{equation*}
which proves the result.
\end{proof}
\begin{remark} \label{rem:pos}
When $z$ is a positive equilibrium, $z^T K z >0$. Therefore, $D_x g(z)$ has at least one negative eigenvalue $- z^T K z$. It is tempting to conclude from this that a positive equilibrium can never be a repeller. However, this is not correct because the eigendirection corresponding to $-z^T K z$ is $\spn \{z\}$, which is not along the simplex $S_{n} = \left \{(x_1, \ldots, x_n) \in \R^n | x_i \ge 0, \sum_{i=1}^n x_i = 1 \right \}$. 
\end{remark}
\begin{corollary} \label{thm:detD}
Let $z$ be a positive equilibrium of $(\HH,K)$. Then 
\[
\det\left(D_x g(z) \right) = - \det(K) \prod_{i=1}^n z_i. 
\]
\end{corollary}
\begin{proof}
Let the eigenvalues of $\diag(z) K^T$ be $\left( \lambda_1 := z^T K z, \lambda_2, \ldots, \lambda_n \right)$. Then 
$\det\left(D_x g(z) \right) = - \lambda_1 \ldots \lambda_n = - \det\left( \diag(z) K^T \right) = - \det(\diag(z)) \det(K)$. 
\end{proof}

\begin{theorem} \label{thm:main_stab2}
Suppose that $(\HH,K)$ has a positive equilibrium $z$, i.e. $z$ is a positive solution of $K^T z = \left( z^T K^T z\right) \one$ for the $n \times n$ matrix $K$. 
Then $z$ is linearly stable if and only if $K$ is invertible and the matrix 
\begin{equation} \label{eq:stab_matrix}
\diag\left(z \right) K^T 
\end{equation}
has exactly $n-1$ eigenvalues with negative real part. 
\end{theorem}
\begin{proof}
Let $z$ be a positive solution of $K^T z = \left(z^T K z\right) \one$ for the $n \times n$ matrix $K$. If $K$ is invertible, then $z$ is the unique equilibrium of $(\HH,K)$. If $\diag(z)K^T $ has exactly $n-1$ eigenvalues with negative real part, then by Theorem \ref{thm:eigD} and Remark \ref{rem:pos}, the $n \times n$ matrix $D_x g(z)$ has $n$ eigenvalues with negative real part. It follows that $z$ is linearly stable. 
Conversely, if the positive equilibrium $z$ is linearly stable, then $z$ is isolated and therefore unique by Theorem \ref{thm:uniqueposeq}. So $K$ must be invertible, $D_x g(z)$ has $n$ eigenvalues with negative real part and $ \diag(z) K^T$ has exactly $n-1$ eigenvalues with negative real part. 
\end{proof}

\begin{corollary}
Suppose that $(\HH,K)$ has a unique positive equilibrium $z = (1/n) \one$. 
Then $z$ is linearly stable if and only if $K^T$ has exactly $n-1$ eigenvalues with negative real part. 
\end{corollary}
\begin{proof}
The eigenvalues of $K^T$ are the same as the eigenvalues of $ \diag\left( \one/n \right) K^T $ up to a multiple of the positive number $1/n$, in particular the real parts of the eigenvalues in the two cases have the same signs. 
\end{proof}

\begin{proof}[Proof of Theorem \ref{thm:main_stab}]
Suppose that $(\HH,K)$ has a unique, positive linearly stable equilibrium $z$. By Theorem \ref{thm:main_stab2}, $K$ is invertible, $z$ is a scalar multiple of $\left(K^T\right)^{-1} \one$, i.e.  $\left(K^T\right)^{-1} \one \in \R^n_{> 0}$, and $\diag\left(z \right) K^T$
has exactly $n-1$ eigenvalues with negative real part. This implies that $\diag((K^T)^{-1} \one) K^T$
has exactly $n-1$ eigenvalues with negative real part. Conversely, suppose that $K$ is invertible, $(K^T)^{-1} \one \in \R^n_{> 0}$, and $\diag\left((K^T)^{-1} \one \right) K^T$ has exactly $n-1$ eigenvalues with negative real part. By the first two conditions, $(\HH,K)$ has a unique, positive equilibrium $z$ which is a scalar multiple of $\left(K^T\right)^{-1} \one$. This implies that $\diag(z) K^T$ has exactly $n-1$ eigenvalues with negative real part. By Theorem \ref{thm:main_stab2}, $z$ is linearly stable. 
\end{proof}

\subsection{Stability of Boundary Equilibria}

\begin{lemma} \label{lem:boundary_eig}
If $z$ is an equilibrium of $(\HH,K)$ such that $z_i =0$, then $f_i(z) - \bar \rho(z)$ is an eigenvalue of $D_x g(z)$. 
\end{lemma}
\begin{proof}
From \eqref{eq:jac_comp}, if $z_i=0$, then $\pdv{g_i}{x_j} |_{x=z} = \delta_{ij} (f_i(z) - \bar \rho (z))$, i.e. the $i$th row of $D_x g(z)$ has zeros everywhere except in the $i$th entry which is $f_i(z) - \bar \rho(z)$. This means the characteristic polynomial $\det(D_x g(z) - \lambda I)$ has the factor $\lambda - (f_i(z) - \bar \rho(z))$, from which the result follows. 
\end{proof}

\begin{theorem} \label{thm:eigsofjac}
Consider a hyperchain system $(\HH,K)$ with the associated dynamical system $\dot x = g(x) = x \ast \left(f(x) - (x \cdot f) \one \right)$. 
Let $\II$ be a nonempty subset of $\{1, \ldots, n\}$ and $\II^c$ its complement. Consider a boundary of $S_n$ defined by $\BB_\II := \{x ~|~ x_i = 0  \iff i \in \II\}$. Let $z \in \BB_\II$ be an equilibrium of $(\HH,K)$ and define $\hat z_i = z_i$ for $i \in \II^c$ and let $\hat z = (z_i : i \in \II^c)$. So the dynamical system associated with $(\HH,K)[X_{\II^c}]$ by $\dot {\hat x} = \hat g(\hat x)$. 
Then the eigenvalues of the Jacobian $D_xg |_{x=z}$ are 
\[
\bigcup_{i \in \II} \left( f_i(z) - (z \cdot f(z)\one)\right) \bigcup \sigma_{\left( (\HH,K)[\II^c] \right)} (\hat z) 
\] 
where $\sigma_{\left( (\HH,K)[\II^c] \right)}(\hat z)$ denotes the set of eigenvalues of the Jacobian matrix of $(\HH,K)[\II^c]$ at its positive equilibrium $\hat z$. 
\end{theorem}
\begin{proof}
Let $z \in \BB_\II$ be an equilibrium of $(\HH,K)$. By Lemma \ref{lem:boundary_eig}, $f_i(z) - \bar \rho(z)$ is an eigenvalue of $D_xg(z)$, and 
\[
\det(\lambda I - D_xg(z)) = \prod_{i \in \II} \left(\lambda - (f_i(z) - \bar \rho(z))\right) \det(\lambda I - D_{\hat x}\hat g(\hat z))
\]
The result follows from this. 
\end{proof}

\begin{corollary}
Consider a hyperchain system $(\HH,K)$ and let $\II$ be a nonempty subset of $\{1, \ldots, n\}$. 
In the notation of Lemma \ref{lem:notational}, suppose that $\hat z$ is an exponentially stable equilibrium of $(\HH,K)[\II^c]$. Then $z$ is an exponentially stable equilibrium of $(\HH,K)$ in $\BB_\II$ if and only if 
\[
\bar \rho (z) > \max_{i \in \II}  f_i(z) 
\]
\end{corollary}
\begin{proof}
Since $\hat z$ is an exponentially stable equilibrium of $(\HH,K)[\II^c]$, all eigenvalues of $D_{\hat x}\hat g(\hat z)$ have negative real part. By Theorem \ref{thm:eigsofjac}, all the remaining eigenvalues have negative real part if and only if the condition $\bar \rho (z) > \max_{i \in \II}  f_i(z)$ is satisfied. 
\end{proof}

\section{Permanence of a hyperchain system} \label{sec:permanence}

A permanent dynamical system must have an interior equilibrium, see \cite{srzednicki1985rest}. We state the result here in the context of a hyperchain system. 

\begin{lemma}
Let $(\HH,K)$ be a hyperchain system. If $(\HH,K)$ is permanent, then $(\HH,K)$ has a positive equilibrium. 
\end{lemma}

It is shown in \cite{hofbauer1988theory} that the $n$-hypercycle is permanent and has a unique positive equilibrium for all rate constants for any $n \ge 2$. For $n=2,3,4$, the unique positive equilibrium is globally stable, while for $n \ge 5$, the positive equilibrium is unstable. 

In this section, we prove partial converses to two known results relating a hyperchain and its permanence properties. We start with the following result from Section 20.3 of \cite{hofbauer1988theory}. 

\begin{theorem}[Hofbauer, Sigmund] \label{thm:perm_irr}
If a hyperchain system $(\HH,K)$ is permanent, then $\HH$ is  strongly connected. Furthermore, if $n \le 5$, then $\HH$ is Hamiltonian. 
\end{theorem}
\begin{example} \label{ex:six}
It is shown in \cite{hofbauer1988theory} that the above statement is not true if the assumption of $n \le 5$ is dropped. We reproduce the example here for convenience. 

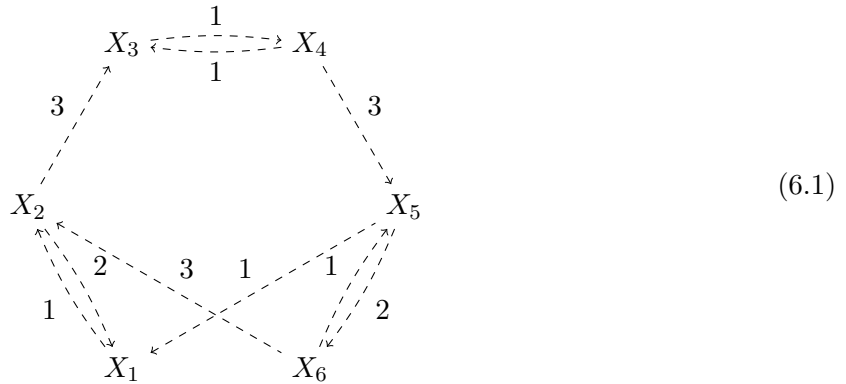
\begin{equation}
\begin{tikzpicture}[baseline={(current bounding box.center)}, scale=2.5]
   \node[state] (x5)  at (1,0)  {$X_5$};
   \node[state] (x4)  at (cos{60},sin{60})  {$X_4$};
   \node[state] (x3)  at (cos{120},sin{120})  {$X_3$};
      \node[state] (x2)  at (-1,0)  {$X_2$};
         \node[state] (x1)  at (cos{240},sin{240})  {$X_1$};
      \node[state] (x6)  at (cos{300},sin{300})  {$X_6$};

   \path[->,dashed]
    (x1) edge[bend left] node {$1$} (x2)
    (x2) edge[bend left] node {$2$} (x1)
    (x2) edge[] node {$3$} (x3)
    (x3) edge[bend left] node {$1$} (x4)
    (x4) edge[bend left] node {$1$} (x3)
    (x4) edge[] node {$3$} (x5)
    (x5) edge[bend left] node {$2$} (x6)
    (x6) edge[bend left] node {$1$} (x5)
    (x2) edge[<-] node {$3$} (x6)
    (x1) edge[<-] node {$1$} (x5)
;
  \end{tikzpicture}
\end{equation}
\end{example}
In fact, the statement in Theorem \ref{thm:perm_irr} can be strengthened by assuming weaker hypotheses. 
\begin{theorem} \label{thm:pers_irr}
Suppose that $(\HH,K)$ is persistent and has a unique positive equilibrium. Then $\HH$ is  strongly connected. 
\end{theorem}
The proof by Hofbauer \& Sigmund, in Section 20.3 of \cite{hofbauer1988theory}, of Theorem \ref{thm:perm_irr} holds almost verbatim for Theorem \ref{thm:pers_irr} after weakening the hypotheses. We omit reproduction of the proof here, while leaving a note for the reader that the matrix $\vect A$, with entries $a_{jk}$, in \cite{hofbauer1988theory} is transpose of the matrix $K$ in this article. 

It is well-known that if $\HH$ is a cycle then $\HH$ is permanent, i.e. $(\HH,K)$ is permanent for all positive rate constants $K>0$ (see \cite{schuster1979dynamical}). We show the converse.

Let $e_i$ be the unit vector with $n$ components whose $i$th component is $1$ and the other components are zero. Let $\DD$ be a linear digraph. Let $i_s$ ($i_p$) denote the succeeding (preceding) node of node $i$ in $\DD$. In other words, $i_s$ is the unique node in $\DD$ such that $i \to i_s \in \DD$, and $i_p$ is the unique node in $\DD$ such that $i_p \to i \in \DD$. Then the $i$th row of $A(\DD)$ is $e_{i_s}$ and the $i$th column of $A(\DD)$ is $e_{i_p}$.

\begin{lemma} \label{lem1}
Let $\DD$ be a linear digraph on $n$ nodes. Let $A(\DD)$ be the adjacency matrix of $\DD$. Then the following hold: 
\been
\item $A(\DD)^T e_i = e_{i_s}$. 
\item $A(\DD)^T \one = \one$. 
\enen
\end{lemma}
\begin{proof}
Note that the $j$th row of $A(\DD)^T$ is $e_{j_p}$ and $e_{j_p} \cdot e_i = \delta_{j_pi}$ which is equal to $1$ if and only if $j_p=i$ which is equivalent to $j=i_s$. 
Since $\DD$ is a linear graph, $A(\DD)^T$ produces a permutation of $e_i$ and so $\one = \sum_{i=1}^n e_i$, $A(\DD)^T \one = \sum_{i=1}^n A(\DD)^T e_i = \one$.
\end{proof}

\begin{theorem} \label{thm:perm_then_cycle}
If $\HH$ is permanent then $\HH$ is cyclic.
\end{theorem}
\begin{proof}
We prove the contrapositive, i.e. if $\HH$ is not cyclic then $\HH$ is not permanent. Suppose $\HH$ is a non-cyclic hyperchain on $n$ species. 

Suppose that $\HH$ does not have a  spanning linear subgraph. Then for any $K$, $(\HH,K)$ either has no positive equilibria or has infinitely many positive equilibria. In either case, $(\HH,K)$ is not permanent. If $\HH$ is not strongly connected, then $(\HH,K)$ is not permanent for any $K$, by the Theorem in Section 20.3 of \cite{hofbauer1988theory}. 

So we assume that $\HH$ is  strongly connected and has a  spanning linear subgraph. Let $\CC_1, \ldots, \CC_\ell$ be the set of cycles that constitute a  spanning linear subgraph of $\HH$. If $\ell=1$, i.e. if $\HH$ is Hamiltonian, $\HH$ must have at least one edge that is not in $\CC_1$, because otherwise $\HH$ is cyclic. If $\ell>1$, there must be at least one edge in $\HH \setminus \cup_{k=1}^\ell \CC_k$, because otherwise $\HH$ is disconnected which means that $\HH$ is not strongly connected. 
Let $\ol{\CC} :=\cup_{k=1}^\ell \CC_k$, and let $e$ be an edge not in $\ol{\CC}$. 
Note that $\HH = \ol{\CC} \dot\cup \{e\} \dot\cup \left(\HH \setminus \left(\ol{\CC} \cup \{e\} \right)\right) =: \HH_1 \dot\cup \HH_2 \dot\cup \HH_3$, where the last set may possibly be empty.  By construction, both $\HH_1$ and $\HH_1 \cup \HH_2$ are sub-hyperchains on $n$ species. 

Define $K_\ep$ as follows: all edges in $\HH_1$ have rate constant $1$, all edges in $\HH_2$ have rate constant $2$, and all edges in $\HH_3$ have rate constant $\ep$. Label the species such that the unique edge in $\HH_2$ is $X_a \to X_b$. Let $c$ be the label of the unique vertex such that $X_c \to X_b \in \ol{\CC}$. 

We first show that the sub-hyperchain system $(\HH, K_0) = (\HH_1 \cup \HH_2, K_0)$ does not have a positive equilibrium. We do this by showing that the equation $K_0^T x = (x^T K_0 x) \one$ has a unique solution and that this solution is outside the nonnegative orthant.
By interchanging rows, $K_0$ can be brought into a triangular form with all diagonal entries equal to $1$, so $\det(K_0)= \pm 1$ which implies that $K_0 $ is invertible and the equation $K_0^T x = (x^T K_0 x) \one$ has a unique solution. Now we show that this unique solution is 
 $z := \frac{1}{n-2}\left(\one - 2 e_c\right) \notin \R^n_{\ge 0}$. By Lemma \ref{lem1}, $K_0^T |_{\HH_1} z = \frac{1}{n-2}\left(\one - 2 e_b\right)$.  The only nonzero entry of $K_0^T |_{\HH_2}$ is the $(b,a)$ entry which is $2$ and so $K_0^T |_{\HH_2} z = \frac{1}{n-2}2e_b$. 
Thus, $K_0^T z = \frac{1}{n-2}\one$ and $z^T K_0^T z = \frac{1}{n-2}$, which implies that $K_0^T z = \left( z^T K_0^T z \right) \one$. Thus $(\HH,K_0)$ does not have a positive equilibrium and is therefore not permanent. 

By continuity of solutions with respect to the parameter $\ep$, it follows that $(\HH,K_\ep)$ does not have a positive equilibrium for sufficiently small $\ep$ and so $\HH$ is not permanent. 
\end{proof}
The Hamiltonian property of a graph is a stronger property than irreducibility for capturing the notion of positive catalytic influence that permeates through the network. While irreducibility does not guarantee existence of a permanent system, we prove in the next theorem that the Hamiltonian property does. The proof is an adaptation of the proof that an $n$-hypercycle is permanent, which appears as Theorem 12.3.1 in \cite{hofbauer1998evolutionary}. 
\begin{theorem} \label{thm:hamilton_perm}
If $\HH$ is Hamiltonian, then there is a $K>0$ such that $(\HH,K)$ is permanent. 
\end{theorem}
\begin{proof}
Let $\HH$ be a hyperchain on $n$ species and let $\CC$ be a Hamiltonian cycle of $\HH$. Label the species along this cycle $(X_1, \ldots, X_n)$. Define $K$ as follows: $k_{ij} = 1$ if $X_i \dasharrow X_j \in \CC$ and $k_{ij} = \frac{1}{4n}$ if $X_i \dasharrow X_j \in \HH \setminus \CC$. 
The mass action dynamical system for $(\HH, K)$ is
\begin{align*}
\dot x_i = x_i \left( x_{i-1} +  \sum_{j \ne i-1} k_{ji} x_j - \rho(x) \right) \quad \mbox{ for } 1 \le i \le n,
\end{align*}
where $\rho(x) = \sum_{i=1}^n x_i \left( x_{i-1} +  \sum_{j \ne i-1} k_{ji} x_j \right)$ and $k_{ji}$ for $j \ne i-1$ is either $\frac{1}{4n}$ or $0$ depending on whether $X_j \dasharrow X_i  \in \HH$ or not. 
Let $P(x) = \prod_{i=1}^n x_i$ and so 
\[
\dv{P(x)}{t} = \sum_{i=1}^n \pdv{P(x)}{x_i} \dv{x_i}{t} = P(x) \sum_{i=1}^n \frac{1}{x_i} \dv{x_i}{t}, 
\]
which implies that 
\[
\Psi(x) := \frac{\dot P(x)}{P(x)} = \sum_{i=1}^n \left( x_{i-1} +  \sum_{j \ne i-1} k_{ji} x_j - \rho(x) \right) = 1- n \rho(x)+ \sum_{i=1}^n  \sum_{j \ne i-1} k_{ji} x_j  
\]
We use Theorem 13.2.1 of \cite{hofbauer1988theory} which states that, if for all $x \in \bd S_n$ there is a $T>0$ such that $\int_0^T \Psi(x(t)) dt > 0$ then the system is permanent. Note that 
\begin{align*}
&\frac{1}{T} \int_0^T \Psi(x(t)) dt = \frac{1}{T} \int_0^T \left(1- n \rho(x) + \sum_{i=1}^n  \sum_{j \ne i-1} k_{ji} x_j \right) dt >0 \\
\iff & \frac{1}{T} \int_0^T \rho(x(t)) dt < \frac{1}{nT} \int_0^T \left(1+ \sum_{i=1}^n  \sum_{j \ne i-1} k_{ji} x_j \right) dt
\end{align*}
Since the right hand side is greater than $\frac{1}{nT} \int_0^T 1 dt = \frac{1}{n}$, it suffices to show that $ \frac{1}{T} \int_0^T \rho(x(t)) dt < \frac{1}{n}$ for all $x \in \bd S_n$ and for some $T>0$.  By way of contradiction, suppose that  there is an $x \in \bd S_n$ and a $T>0$ such that $\frac{1}{T} \int_0^T \rho(x) dt \ge \frac{1}{n}$ where $x(0) = x$. Since $x \in \bd S_n$, there is an index $m \in \{1,\ldots,n\}$ such that $x_m(t) = 0$ for $t \ge 0$. We show by way of induction that if $x_i(t) \to 0$ then $x_{i+1}(t) \to 0$. Assume that $x_i(t) \to 0$ and suppose that $x_{i+1}(t) > 0$. Then 
\begin{align*}
\dv{}{t} \log x_{i+1}(t) = \frac{\dot x_{i+1}(t)}{x_{i+1}(t)} = x_i(t)  + \frac{1}{4n} \sum_{\substack{X_j \dashrightarrow X_{i+1} \in \HH : j \ne i}}  x_j(t) - \rho(x(t)).
\end{align*}
Integrating from $0$ to $T$ and dividing by $T$, 
\begin{align*}
& \frac{1}{T} \left(\log(x_{i+1}(T)) - \log(x_{i+1}(0)) \right)  = \frac{1}{T} \int_0^T x_i(t) dt + \frac{1}{4nT} \int_0^T \sum_{\substack{X_j \dashrightarrow X_{i+1} \in \HH : \\ j \ne i}}  x_j(t) dt - \frac{1}{T} \int_0^T \rho(x) dt \\
& \le \frac{1}{T} \int_0^T x_i(t) dt + \frac{1}{4nT} \int_0^T 1 dt - \frac{1}{T} \int_0^T \rho(x) dt 
 \le \frac{1}{T} \int_0^T x_i(t) dt + \frac{1}{4n}  - \frac{1}{n}. 
\end{align*}
Since $x_i(t) \to 0$, $\frac{1}{T} \int_0^T  x_i(t) dt < \frac{1}{4n}$ for sufficiently large $T$, 
$
\frac{1}{T} \left(\log(x_{i+1}(T)) - \log(x_{i+1}(0)) \right) \le \frac{1}{4n} + \frac{1}{4n} - \frac{1}{n} = -\frac{1}{2n}, 
$
which implies that $x_{i+1}(T) < x_{i+1}(0) e^{-T/2n} \to 0$. Therefore, by induction $x_j(t) \to 0$ for all $j$ which contradicts $\sum_i x_i(t)=1$. So for all $x \in \bd S_n$ there is a $T>0$ such that $\int_0^T \Psi(x(t)) dt > 0$, which in turn implies that $(\HH,K)$ is permanent. 
\end{proof}
The results on the dynamics of relative concentrations in a hyperchain system are summarized in Figure \ref{fig:sum_implications}.

\subsection*{Acknowledgments}
The authors thank the American Institute of Mathematics (AIM) for hosting the SQuaRE
workshop ``Dynamical properties of deterministic and stochastic models of reaction networks'' in March 2019, where the discussion on this topic was initiated. 
BJ is grateful to the Department of Mathematics, University of Wisconsin-Madison for hosting a sabbatical visit in Spring 2019. GC thanks the National Science Foundation for support through the DMS-1816238 grant, and the Simons Foundation for support through Simons Fellows in Mathematics. Thanks to Andr{\'e} K{\"u}ndgen for helpful comments.

\bibliographystyle{unsrt}
\bibliography{hc}

\end{document}